\documentclass[10pt]{amsart}
\usepackage[latin1]{inputenc}
\usepackage[T1]{fontenc}
\usepackage{amsmath,amscd}
\usepackage{graphicx}
\usepackage[all]{xy}
\usepackage{amsthm}
\usepackage{amssymb,url}
\usepackage[charter]{mathdesign}
\usepackage{amsfonts}
\usepackage{mathrsfs}
\usepackage{amsxtra}
\usepackage{bbm}
\usepackage{ae}
\usepackage[all]{xy}
\usepackage{color}

\newtheorem{thm}{Theorem}[section]
\newtheorem{thm*}{Theorem}
\newtheorem{lem}[thm]{Lemma}
\newtheorem{cor}[thm]{Corollary}
\newtheorem{cor*}[thm*]{Corollary}
\newtheorem{prop}[thm]{Proposition}

\theoremstyle{remark}
\newtheorem{remark*}[thm*]{Remark}
\newtheorem{remark}[thm]{Remark}

\theoremstyle{definition}
\newtheorem{deef}[thm]{Definition}

\newcommand{\R}{\mathbbm{R}}
\newcommand{\C}{\mathbbm{C}}

\newcommand{\Z}{\mathbbm{Z}}

\newcommand{\N}{\mathbbm{N}}

\newcommand{\id}{ \mathrm{id}}
\newcommand{\rd}{\mathrm{d}}

\newcommand{\Bo}{\mathcal{B}}

\newcommand{\He}{\mathcal{H}}

\newcommand{\Fg}{\mathcal{F}}

\newcommand{\Dom}{\mathrm{Dom}\,}

\newcommand{\wind}{\mathrm{wi} \mathrm{n} \mathrm{d} \,}

\renewcommand{\epsilon}{\varepsilon}

\newcommand{\e}{\mathrm{e}}
\newcommand{\tra}{\mathrm{t}\mathrm{r}}

\title{Spectral flow of exterior Landau-Robin hamiltonians}
\author{Magnus Goffeng, Elmar Schrohe}

\begin{document}
\maketitle

\begin{abstract}
We study the spectral flow of Landau-Robin hamiltonians in the exterior of a compact domain with smooth boundary. This provides a method to study the spectrum of the exterior Landau-Robin hamiltonian's dependence on the choice of Robin data, even explaining the heuristics of how the spectrum of the Robin problem asymptotically tends to the spectrum of the Dirichlet problem. The main technical result concerns the continuous dependence of Landau-Robin hamiltonians on the Robin data in the gap topology. The problem can be localized to the compact boundary where the asymptotic behavior of the spectral flow in some special cases can be described.  
\end{abstract}

\large
\section*{Introduction}
\normalsize

In this paper we initiate the study of how the spectrum of the exterior Landau-Robin hamiltonian depends on the choice of Robin data by means of spectral flow. Technical issues aside, the spectral flow counts the number of eigenvalues that cross a point in the spectrum, taking the direction of the crossing into account. As such, the spectral flow measures how the spectrum ``moves" under a change of Robin data. The study of spectral flow has proven useful in noncommutative topology, where it describes the odd index pairing relating it to the index theory of Toeplitz operators. Spectral flow was used by Atiyah-Patodi-Singer \cite{APS3}, in joint work with Lusztig, to describe the variations of the spectral boundary contributions in the index formula now known as the Atiyah-Patodi-Singer index theorem. Atiyah-Lusztig's notion of spectral flow was developed further by Phillips \cite{phillie}. In the spirit of index theory, the paper aims at reducing the problem of computing spectral flows to a problem on the compact boundary of the domain where the exterior problem is defined. 

An important application of exterior, as well as interior, magnetic hamiltonians is for instance in the Ginzburg-Landau theory of superconductors, describing Bose-Einstein condensates, see \cite{AftaHel, FouHell}. Other applications of magnetic edge states can be found in \cite{hornbergsmilie}, where the spectrum of Landau hamiltonians in the exterior of compact domains was studied. The spectral theory of exterior Landau hamiltonians was to the authors' knowledge first studied in the mathematics literature in \cite{pushroz} for Dirichlet conditions. Similar results were obtained for the case of Neumann conditions in \cite{mps}. 

In both the Dirichlet and Neumann case, the spectrum clusters in a super exponential fashion around the spectrum of the Landau hamiltonian (without any obstacle). The difference between Dirichlet conditions and Neumann conditions being that in the former case, the clustering takes place to above while in the latter it clusters to below the spectrum of the Landau hamiltonian. Physically, Dirichlet conditions correspond to an infinite potential barrier in the compact obstacle which pushes up the energy while Neumann conditions correspond to a perfect insulator lowering the total energy in the system.

Intermediately between Dirichlet and Neumann conditions, there are Robin conditions -- formally, Dirichlet conditions are obtained by letting the Robin data tend to infinity. It was proven in \cite{gkp} that the spectral behavior of exterior Landau-Robin hamiltonians resembles that of exterior Landau-Neumann hamiltonians. One of the motivating problems for this paper is the formalizing of the procedure described above in regards to seeing spectral properties of the Landau-Dirichlet hamiltonian as a limit case of Landau-Robin operators.

\subsection{Setup} 

The Landau hamiltonian with magnetic field strength $b\in \R^\times$ as a differential expression is given by the second order elliptic operator on $\R^{2d}$ defined as
\begin{equation}
\label{lbexp}
L_b:=-(\nabla+ibA_0)^2, \quad\mbox{where}\quad A_0(x_1,x_2,\ldots, x_{2d})=\frac{1}{2}\left(-x_2,x_1,\ldots , -x_{2d}, x_{2d-1}\right).
\end{equation}
Since $\overline{L_b}=L_{-b}$ it suffices to consider $b>0$. We often suppress the $b$-dependence by writing $A=bA_0$. The operator $L_b$ models $d$ uncoupled particles moving in $\R^2$ under the influence of a constant perpendicular magnetic field of strength $b$. The choice of $A_0$ is non-physical, but greatly simplifies the analysis. The differential expression \eqref{lbexp} equipped with the domain $C^\infty_c(\R^{2d})$ defines an essentially self-adjoint operator on $L^2(\R^{2d})$. By an abuse of notation we also let $L_b$ denote the closure of this operator. The domain of $L_b$ is the magnetic Sobolev space $H^2_{A}(\R^{2d})$, where
$$H^k_{A}(\R^{2d}):=\{u\in L^2(\R^{2d}): \; (\nabla+iA)^ju\in L^2(\R^{2d}), \, j=0,\ldots, k\},\quad\mbox{for}\; k\in \N.$$
The spectrum of $L_b$ has been known already since the work of Fock \cite{fock}, and rediscovered by Landau \cite{landau} a few years later. The spectrum of $L_b$ is $\sigma_L:=2b\N+bd$, each point being an eigenvalue of infinite multiplicity, for details, see for instance \cite{rozta}. The eigenspace corresponding to a point $\Lambda_q=2b(q-1)+bd$, for $q\in \N_{>0}$ in the spectrum of $L_b$ is referred to as the $q$-th Landau level. We use the standard convention $\Lambda_0=-\infty$.

We consider a compact domain $K\subseteq \R^{2d}$ with smooth boundary and set $\Omega:=\R^{2d}\setminus K$. The operator $L_b|_{C^\infty_c(\Omega)}$ is not essentially self-adjoint on $L^2(\Omega)$. We will concern ourselves with different self-adjoint extensions of this operator. The \emph{Dirichlet realization} $L_{b,D}^\Omega$ is the differential expression $L_b$ equipped with the domain 
$$\Dom L_{b,D}^\Omega:= H^2_{A,0}(\Omega):=\left\{u\in H^2_{A}(\Omega): u|_{\partial \Omega}=0\right\}.$$
The magnetic Sobolev spaces can be defined for any domain $\Omega$, and standard elliptic regularity estimates show that $H^k_{A}(\Omega)$ locally coincides with $H^k(\R^d)$. Since $\Omega$ has a smooth boundary, the trace operator $\gamma_{\partial\Omega}:H^k_A(\Omega)\to H^{k-1/2}(\partial\Omega)$ is continuous for any $k\geq 1$. Letting $\nu_\Omega$ denote the unit outward normal to $\partial \Omega$, we set $\partial_N:=\nu_\Omega\cdot (\nabla+ibA_0)$ -- the magnetic Neumann operator. 
\label{neumandef}
The operator $\partial_N$ acts continuously $H^k_{A}(\Omega)\to H^{k-3/2}(\partial \Omega)$ for $k\geq 2$. For any self-adjoint pseudo-differential operator $\tau\in \Psi^0(\partial \Omega)$, we consider the \emph{Robin realization} $L^\Omega_{b,\tau}$ given by equipping $L_b$ acting on $H^2_{A}(\Omega)$ with the domain
$$\Dom(L_{b,\tau}^\Omega):=\{u\in H^2_{A}(\Omega): \partial_Nu +\tau \gamma_{\partial \Omega}(u)=0\}.$$
The Landau-Robin hamiltonian $L^\Omega_{b,\tau}$ can also be realized as the self-adjoint operator associated with the quadratic form defined on $H^1_A(\Omega)$ by
$$\mathfrak{q}_{b,\tau}^\Omega[u]=\int_\Omega |(\nabla+ibA_0)u|^2\rd V+\int_{\partial \Omega} \tau(u|_{\partial \Omega})\overline{u|_{\partial \Omega}}\rd S.$$
By \cite{pushroz}, $\sigma_{ess}(L^\Omega_{b,D})=\sigma(L_b)=2b\N+bd$. The same identity holds in the Neumann case by \cite{mps} and in the Robin case by \cite{kapao}, see also \cite{gkp} for the Robin case. Thus, the study of how the Robin-Landau hamiltonian depends on its Robin data $\tau$ reduces to studying finite-dimensional eigenvalues --  whose change the spectral flow measures.

\subsection{Main results}

The main results of this paper are concerned with the spectral dependence of $L^\Omega_{b,\tau}$ on $\tau$. The main technical tool in this direction is the following theorem. Similar to \cite{bobaleph}, we use the notation $\mathcal{CF}^{s.a.}(L^2(\Omega))$ for the space of closed self-adjoint Fredholm operators defined in $L^2(\Omega)$. We equip $\mathcal{CF}^{s.a.}(L^2(\Omega))$ with the gap-topology -- a topology defined from the metric 
$$d_{gap}(T_1,T_2):=\|(T_1+i)^{-1}-(T_2+i)^{-1}\|_{\Bo(L^2(\Omega))}.$$ 
For more details, see below in Subsection \ref{sfsubsection} or \cite{bobaleph}. We also let $\Psi^0(\partial\Omega)^{s.a.}$ denote the real subspace of self-adjoint elements in $\Psi^0(\partial \Omega)$. We will throughout the paper use the notation 
$$\sigma_L:=2b\N+bd.$$

\begin{thm*}
\label{gapconthm}
When equipping $\Psi^0(\partial\Omega)^{s.a.}$ with the topology induced from the norm topology of $\Bo(H^{3/2}(\partial\Omega),H^{1/2}(\partial\Omega))$, any $\mu\in \R\setminus \sigma_L$ gives a continuous mapping 
$$\Psi^0(\partial\Omega)^{s.a.}\to \mathcal{CF}^{s.a.}(L^2(\Omega)), \quad \tau\mapsto L^\Omega_{b,\tau}-\mu.$$ 
\end{thm*}

The proof of Theorem \ref{gapconthm} will occupy Subsection \ref{provingmainthm}. It is based on standard ideas from boundary value problems presented in Section \ref{fundsect}. Some care with the technical details is needed because $\Omega$ is not compact. Theorem \ref{gapconthm} holds also for $\tau$ ranging over $\Psi^t(\partial \Omega)$ for any $t<1$. The proof for $0<t<1$ proceeds mutatis mutandis from the case $t=0$ using non-classical pseudo-differential operators; we avoid this case for simplicity.

\begin{remark*}
Already at this point, we emphasize that the smoothing finite rank operators on $L^2(\partial\Omega)$ are dense in $\Psi^0(\partial\Omega)$ in the norm topology of $\Bo(H^{3/2}(\partial\Omega),H^{1/2}(\partial\Omega))$, making it possible to reduce to the finite rank case (see Remark \ref{smoothingappsremk} below). The precise explanation for the appearance of the topology coming from $\Bo(H^{3/2}(\partial\Omega),H^{1/2}(\partial\Omega))$ is found in Lemma \ref{gapestimates}. 
\end{remark*}

\begin{remark*}
The proof of Theorem \ref{gapconthm} only uses the fact that for $\mu\in \C\setminus \R$, $L_b+\mu$ has a fundamental solution $E_\mu\in C^\infty(\R^{2d}\times \R^{2d}\setminus \Delta_{\R^{2d}})$ such that the associated single and double layer potentials on $\Omega$ restrict to pseudo-differential operators of order $-1$ on $\partial\Omega$ and give bounded mappings $H^{1/2}(\partial\Omega)\to \Dom(L_b)$ and $H^{3/2}(\partial\Omega)\to \Dom(L_b)$, respectively. For $L_b$, this is the content of Proposition \ref{singcontrest} and Lemma \ref{singcont}, respectively. As such, Theorem \ref{gapconthm} holds in full generality for the exterior of a compact smooth domain in a Riemannian manifold when replacing $L_b$ with a Bochner-Laplacian having a self-adjoint extension to $L^2$ and a fundamental solution satisfying the above properties. In this context, $\sigma_L=2b\N+bd$ is replaced by the essential spectrum of the Bochner-Laplacian at hand.
\end{remark*}

An immediate consequence of Theorem \ref{gapconthm} and the results of \cite{bobaleph,gkp} is the following corollary.

\begin{cor*}
\label{sfcor}
For a path $(\tau_t)_{t\in [0,1]}\subseteq \Psi^0(\partial \Omega)^{s.a.}$ continuous in the norm topology of $\Bo(H^{3/2}(\partial\Omega),H^{1/2}(\partial\Omega))$ and a $\mu\in \R\setminus\sigma_L$, the spectral flow $\mathrm{sf}\,(L^\Omega_{b,\tau_t}-\mu)_{t\in [0,1]}$ is well defined and depends {\bf only} on $b$, $\Omega$, $\mu$ and the end-points $\tau_0$ and $\tau_1$ of $(\tau_t)_{t\in [0,1]}$. Moreover, if $\mu\notin \sigma(L^\Omega_{b,\tau_1})$, for $b$, $\Omega$, $\mu$ and $\tau_0$ fixed, the spectral flow $\mathrm{sf}\,(L^\Omega_{b,\tau_t}-\mu)_{t\in [0,1]}$ is constant in a neighborhood of $\tau_1$ in the $\Bo(H^{3/2}(\partial\Omega),H^{1/2}(\partial\Omega))$-topology.
\end{cor*}

The definitions and properties of spectral flows will be recalled below in Subsection \ref{sfsubsection}. 

\begin{remark*}
\label{smoothingappsremk}
The fact that $\mathrm{sf}\,(L^\Omega_{b,\tau_t}-\mu)_{t\in [0,1]}$ is constant in a neighborhood of $\tau_1$ in the $\Bo(H^{3/2}(\partial\Omega),H^{1/2}(\partial\Omega))$-topology if $\mu\notin \sigma(L^\Omega_{b,\tau_1})$ can be used as follows. After picking an $L^2(\partial\Omega)$-orthonormal eigenbasis  $(e_k)_{k\in \N}$ of a positive order self-adjoint elliptic pseudodifferential operator $D$ on $\partial \Omega$ we can approximate any pseudo-differential operator $\tau$ by the finite-rank smoothing operator
$$T_{(N)}(\tau):=\sum_{j,k=0}^N\langle \tau e_k,e_j\rangle_{L^2(\partial\Omega)}e_j\otimes e_k^*,$$
where $e_j\otimes e_k^*$ denotes the rank one operator $f\mapsto \langle f,e_k\rangle e_j$. For a constant $C>0$, depending only on $D$, we have the estimate 
\begin{equation}
\label{tauapr}
\|\tau-T_{(N)}(\tau)\|_{\Bo(H^{3/2}(\partial\Omega),H^{1/2}(\partial\Omega))}\leq C\left(\|\tau\|_{\Bo(H^{1/2}(\partial\Omega))} +\|\tau\|_{\Bo(H^{3/2}(\partial\Omega))} \right)N^{-\frac{1}{2d-1}}.
\end{equation}
For a proof of the estimate \eqref{tauapr} see Proposition \ref{proooooftauapr} on page \pageref{proooooftauapr}. Therefore, for $N$ large enough, we can define a path $\tilde{\tau}_t:=\tau_0+tT_{(N)}(\tau_1-\tau_0)$ and from Corollary \ref{sfcor} deduce 
$$\mathrm{sf}\,(L^\Omega_{b,\tau_t}-\mu)_{t\in [0,1]}=\mathrm{sf}\,(L^\Omega_{b,\tilde{\tau}_t}-\mu)_{t\in [0,1]},$$
reducing the computation of the spectral flow to a spectral flow along a finite-rank perturbation. We also note that if $\mu\notin \sigma(L^\Omega_{b,\tau_0})\cup \sigma(L^\Omega_{b,\tau_1})$, we can for $N$ large enough write 
$$\mathrm{sf}\,(L^\Omega_{b,\tau_t}-\mu)_{t\in [0,1]}=\mathrm{sf}\,(L^\Omega_{b,tT_N(\tau_0)}-\mu)_{t\in [0,1]}-\mathrm{sf}\,(L^\Omega_{b,tT_N(\tau_1)}-\mu)_{t\in [0,1]},$$
reducing the computation of the spectral flow even further to the case of finite rank perturbations of the Neumann boundary condition.
\end{remark*}

It is in general quite difficult to compute the spectral flow of Corollary \ref{sfcor}. To simplify matters, we localize the problem to the closed boundary $\partial \Omega$ in Subsection \ref{holoandloc}. We will equip $\Psi^0(\partial\Omega)$ with its usual Fr\'echet topology unless stated otherwise; this topology is stronger than the $\Bo(H^{3/2}(\partial\Omega),H^{1/2}(\partial\Omega))$-topology.

\begin{thm*}
\label{localizintobo}
We set $\sigma_D:=\sigma(L^\Omega_{b,D})$. There is a geometrically defined family, described below in Remark \ref{dirichlrobrem}:
$$\Gamma:\C\setminus\sigma_L\times \Psi^0(\partial\Omega)\to \Psi^{-1}(\partial\Omega),$$
which is holomorphic both in $\mu\in \C\setminus \sigma_L$ and $\tau \in \Psi^0(\partial\Omega)$, such that for any $\tau\in \Psi^0(\partial\Omega)^{s.a.}$
$$\sigma(L^\Omega_{b,\tau})\setminus  \sigma_D=\{\mu\in \C\setminus \sigma_D:\; 1+\Gamma(\mu,\tau)\;\mbox{is not invertible on} \;L^2(\partial\Omega)\}.$$
Furthermore, for any $\mu\in \R\setminus  \sigma_D$ and any path $(\tau_t)_{t\in [0,1]}\subseteq \Psi^0(\partial \Omega)^{s.a.}$ being holomorphic in a neighborhood of $[0,1]\subseteq \C$, 
\begin{align}
\nonumber
\mathrm{sf}\,(L^\Omega_{b,\tau_t}-&\mu)_{t\in [0,1]}\\
\label{sfandgamma}
&=\sum_{t\in Z_\mu(\tau)}\lim_{\epsilon \to 0}\mathrm{sign}\,\frac{\tra_{L^2(\partial\Omega)}\left(\partial_t\Gamma(\mu,\tau_{t+\epsilon})\cdot \Gamma(\mu,\tau_{t+\epsilon})^{d-1}(1+\Gamma(\mu,\tau_{t+\epsilon}))^{-1}\right)}{\tra_{L^2(\partial\Omega)}\left(\partial_\mu\Gamma(\mu,\tau_{t+\epsilon})\cdot \Gamma(\mu,\tau_{t+\epsilon})^{d-1}(1+\Gamma(\mu,\tau_{t+\epsilon}))^{-1}\right)},
\end{align}
where $Z_\mu(\tau)\subseteq \{t\in [0,1]: -1\notin \sigma(\Gamma(\mu,\tau_t))\}$ is a finite set defined below in Remark \ref{summinglocalizintobo}.
\end{thm*}

\begin{remark*}
Each term on the right hand side of Equation \eqref{sfandgamma} is shown below to be well defined for $\epsilon$ in a small neighborhood of $0$ with $0$ removed. The appearance of the spectrum of the Landau-Dirichlet hamiltonian is to guarantee that the boundary value problem the Landau-Robin hamiltonian defines corresponds to an elliptic problem on the boundary, see Lemma \ref{muandlbtau}. Since the spectrum of the Landau-Dirichlet hamiltonian accumulates at the Landau levels from above, and the spectrum of the Landau-Robin hamiltonian accumulates at the Landau levels from below, one can expect the interesting phenomena of Landau-Robin hamiltonians to occur away from the spectrum of the Landau-Dirichlet hamiltonian. 
\end{remark*}

We prove monotonicity results for the spectral flow in Subsection \ref{monoto}. In Theorem \ref{monotonethm} we prove that under a positive change of Robin data, positive in the sense of operators on $L^2(\partial\Omega)$, the spectral flow is non-negative. We also prove a strict monotonicity result assuming a strictly positive change of Robin data and a further spectral condition that can be verified using the Kato-Temple inequality. The above Theorem \ref{localizintobo} can be combined with the monotonicity property of eigenvalues under a change of Robin data leading us to the following asymptotics for the spectral flow.

\begin{thm*}
\label{asymptoticofsf}
For any $\tau\in \Psi^0(\partial \Omega)^{s.a.}$ and $\mu\in \R\setminus \sigma_L$, 
$$\mathrm{sf}\,(L^\Omega_{b,\tau+t}-\mu)_{t\in [0,\gamma]}=\frac{\mathrm{vol}(S^*\partial\Omega)}{(2\pi)^{2d-1}}\gamma^{2d-1}+O(\gamma^{2d-2})\quad\mbox{as}\quad \gamma\to \infty.$$
\end{thm*}

The resemblance between Theorem \ref{asymptoticofsf} and the Weyl law is no coincidence -- its proof consists of a computation using Theorem \ref{localizintobo} reducing $\mathrm{sf}\,(L^\Omega_{b,\tau+t}+\mu)_{t\in [0,\gamma]}$ to the counting function for the Dirichlet-Robin operator on $\partial\Omega$. This Weyl law of the spectral flow follows from the slightly more general statement of Corollary \ref{weyllawforpsido} on page \pageref{weyllawforpsido}. It is discussed in the specific example of the exterior of the disc in $\R^2$ in Subsection \ref{disccompsubse}.

\begin{remark*}
The fact that the spectral flow $\mathrm{sf}\,(L^\Omega_{b,\tau+t}-\mu)_{t\in [0,\gamma]}$ coincides with the spectral counting function of an elliptic pseudo-differential operator on $\partial\Omega$ constructed from $\tau$ implies that there can be no general formula for the spectral flow only depending on the formal symbol of the path $(\tau_t)_{t\in [0,1]}$ in $C^\infty([0,1],\Psi^0(\partial\Omega)/\Psi^{-\infty}(\partial\Omega))$. This observation can also be seen from Remark \ref{smoothingappsremk}.
\end{remark*}

\begin{remark*}
Theorem \ref{asymptoticofsf} asymptotically describes how eigenvalues cross points outside the Landau levels. The heuristics of letting $\gamma\to +\infty$ is that it tends to the Dirichlet condition, a heuristics that can be given meaning to through Theorem \ref{asymptoticofsf}. The latter Theorem formalizes how the clustering of the Landau-Robin operators eigenvalues below the Landau levels move up to above the Landau level where the clusters of the Landau-Dirichlet operator reside. 
\end{remark*}

\large
\section{Operators associated with the fundamental solution}
\label{fundsect}
\normalsize

In this section we will study the properties of a number of operators associated with the fundamental solution of $L_b-\mu$, for $\mu$ outside the spectrum of $L_b$. The operators introduced in this section will play a crucial role in understanding the spectral properties of the Robin operators and the gap continuous dependence on the Robin data.

\subsection{The fundamental solution}
Let $h$ be the positive number solving $\coth(h)=4$. The purpose of $h$ is explained later. For a positive natural number $d$ and $\mu\in \C$ with $\mathrm{Re}(\mu)<d$ we define the smooth functions $I_0(\mu,\cdot), I_\infty(\mu,\cdot)\in C^\infty(\R_{>0})$ by
$$I_0(\mu,s):=\int_0^h\frac{\e^{-s\coth(t)+\mu t}}{\sinh^d(t)}\rd t\quad \mbox{and}\quad I_\infty(\mu,s):=\int_h^\infty\frac{\e^{-s\coth(t)+\mu t}}{\sinh^d(t)}\rd t.$$
We also set $I:=I_0+I_\infty$. It turns out that $I_0$ is entire in $\mu$, but singular as $s\to 0$. On the other hand $I_\infty$ has poles $\mu\in 2\N+d$ but is smooth up to $s=0$.

\begin{lem}[cf. Lemma A.1 of \cite{gkp}]
\label{imustruc}
The functions $I_0$ and $I_\infty$ can be holomorphically extended in $\mu$ to functions in $C^\infty(\C\setminus (2\N+d)\times \R_{> 0})$ satisfying 
\begin{enumerate}
\item $I_\infty$ extends to a smooth function on $\C\setminus (2\N+d)\times \R_{\geq 0}$ satisfying 
$$I_\infty(\mu,s)=O(s^Ne^{-s}), \quad\mbox{as}\quad s\to \infty,$$ 
locally uniformly in $\mu$ for some $N=N(\mu)\in \N$ that grows at most linearly in $|\mu|$.
\item $I_0$ extends to a smooth function on $\C\times \R_{> 0}$ satisfying  
$$I_0(\mu,s)=O(e^{-s}), \quad\mbox{as}\quad s\to \infty,$$ 
locally uniformly in $\mu$
\item There are entire functions $\mathfrak{c}_j,\mathfrak{d}_j\in \mathcal{O}(\C)$ (depending on $d$) such that as $s\to 0$,
$$I_0(\mu,s)=\begin{cases}
(d-2)!\,s^{1-d}+\sum_{j=2-d}^{{+\infty}} \mathfrak{c}_j(\mu)\,s^j +\sum_{j=0}^{+\infty} \mathfrak{d}_j(\mu) \,s^j\log(s),&\mbox{for  $d>1$},\\
\\
\log(s)+\sum_{j=1}^{{+\infty}} \mathfrak{c}_j(\mu)\,s^j +\sum_{j=1}^{+\infty} \mathfrak{d}_j(\mu) \,s^j\log(s) ,&\mbox{for  $d=1$}.
\end{cases}$$
\end{enumerate}
In particular, $I(\mu,s)=O(s^Ne^{-s})$ as $s\to \infty$ locally uniformly in $\mu$ and admits a polyhomogeneous expansion, holomorphically in $\mu$, at $s=0$.
\end{lem}

\begin{proof}
To prove $(1)$, we use the change of variables $\zeta=\coth(t)-1$ showing that
$$I_\infty(\mu,s)=\e^{-s}\int_0^3\e^{-s\zeta}(\zeta+2)^{(d-2+\mu)/2} \zeta^{(d-2-\mu)/2}\rd \zeta.$$
Define the distribution valued function $\mathfrak{f}(\mu,\zeta):=\zeta^{(d-2-\mu)/2}\cdot \chi_{[0,3]}(\zeta)$. It is clear that $\mathfrak{f}(\mu,\cdot)$ is compactly supported and that we can extend $\mathfrak{f}$ to a holomorphic compactly supported distribution valued $\mathfrak{f}\in \mathcal{O}( \C\setminus (2\N+d),\mathcal{E}'(\R))$ whose order is bounded by a linear expression in $|\mu|$. We can find a function $\phi_s\in C^\infty(\R)$ such that $\phi_s(\zeta)=\e^{-s\zeta}(\zeta+2)^{(d-2+\mu)/2}$ near $[0,3]$ where it satisfies the estimate $|\partial_\zeta^j\phi_s(\zeta)|\leq C_j(1+|s|)^j$ uniformly in $s$ for some constant $C_j>0$. Therefore, $I_\infty$ extends by analyticity to a smooth function on $\C\setminus (2\N+d)\times \R_{\geq 0}$ and $I_\infty(\mu,s)=O(s^Ne^{-s})$ where $N$ is the order of $\mathfrak{f}(\mu,\cdot)$.

As for $(2)$, $I_0$ extends by analyticity to $\C\times \R_{>0}$. For $s$ large, by differentiating under the absolutely convergent integral defining $I_0$, $I_0(\mu,s)=O(e^{-s})$ as $s\to \infty$. In proving $(3)$, the precise form of $h$ plays a role. After changing variables $\zeta=\coth(t)$, so that $e^t=(1+2(\zeta-1)^{-1})^{1/2}$, we see that 
$$I_0(\mu,s)=\int_4^\infty \e^{-s\zeta}(\zeta^2-1)^{(d-2)/2}\left(1+\frac{2}{\zeta-1}\right)^{\mu/2}\rd \zeta.$$
For $\zeta\geq \coth(h)= 4$, we can expand $\zeta^{-d}(\zeta^2-1)^{(d-2)/2}(1+2(\zeta-1)^{-1})^{\mu/2}$ in an absolutely convergent Taylor series in powers of $\zeta^{-1}$. Using this, a lengthier computation shows the identity 
\begin{align*}
\int_4^\infty \e^{-s\zeta}(\zeta^2-1)&^{(d-2)/2}\left(1+\frac{2}{\zeta-1}\right)^{\mu/2}\rd \zeta\\
&=\sum_{k,l,m=0}^\infty (-1)^{k+m}2^l\begin{pmatrix} \frac{d-2}{2}\\k\end{pmatrix}\begin{pmatrix} \frac{\mu}{2}\\l\end{pmatrix}\begin{pmatrix} -l\\m\end{pmatrix}s^{2k+l+m+1-d} g_{d-2k-2-l-m}(4s),
\end{align*}
where $g_n(v)=\int_v^\infty \e^{-\zeta}\zeta^n\rd \zeta$ and $\begin{pmatrix} x\\j\end{pmatrix}=\frac{x(x-1)\cdots (x-j+1)}{j!}$ denotes the binomial coefficient. The desired expansion follows from the expansion of $g_n$ proven in \cite[Lemma A.1]{gkp}. 
\end{proof}

We often identify $\R^{2d}$ with $\C^d$, under which $\bar{x}=(x_1,-x_2,\ldots, x_{2d-1},-x_{2d})$. We let $\Delta\subseteq \R^{2d}\times \R^{2d}$ denote the diagonal, the associated distribution $C^\infty_c(\C^d\times \C^d)\ni \varphi\mapsto \int_\Delta \varphi(z,z)\rd V(z)$ we denote by $[\Delta]$. The following proposition follows from \cite{simonfunc}, see also \cite{gkp}.

\begin{prop}
\label{fundsolprop}
For $b>0$ and $\mu \in \C\setminus \sigma_L$, the function $E_{\mu,b}\in C^\infty(\R^{2d}\times\R^{2d}\setminus \Delta)$ defined by 
$$E_{\mu,b}(x,y)=\frac{2b^{d-1}}{(4\pi)^d}\e^{i\mathrm{Im}(b\bar{x}\cdot y)/2}I\left(\frac{\mu}{b},\frac{b|x-y|^2}{4}\right)$$
solves the equation $(L_{b,x}-\mu)E_{\mu,b}=(L_{b,y}-\mu)E_{\mu,b}=[\Delta]$. 
\end{prop}

\begin{remark}
The reader can verify that, for $q\in \N$, the smooth kernel
$$K_{q+1}(x,y)=\mathrm{res}_{\mu=2bq+bd} E_{\mu,b}(x,y)=\frac{2b^{d-1}}{(4\pi)^d}\e^{i\mathrm{Im}(b\bar{x}\cdot y)/2}\cdot \mathrm{res}_{\mu=2bq+bd} I_\infty\left(\frac{\mu}{b},\frac{b|x-y|^2}{4}\right),$$
defines the orthogonal projection onto the $(q+1)$-st Landau level $\ker(L_b-2bq+bd)\subseteq L^2(\R^{2d})$, compare to \cite[Equation (1)]{goffhll}.
\end{remark}

\subsection{Single and double layer potentials}

Let $\Omega\subseteq \R^{2d}$ be a domain with smooth compact boundary, $\nu_\Omega$ the outward pointing normal and $\partial_N:H^k_A(\Omega)\to H^{k-3/2}(\partial\Omega)$ the associated magnetic normal derivative (see Introduction, page \pageref{neumandef}). We define the single layer potential 
$$\mathcal{A}_\mu:C^\infty(\partial\Omega)\to C^\infty(\Omega), \quad \mathcal{A}_\mu f(x):=\int_{\partial\Omega} E_\mu(x,y)f(y)\rd S(y),$$
and the double layer potential
$$\mathcal{B}_\mu:C^\infty(\partial\Omega)\to C^\infty(\Omega), \quad \mathcal{B}_\mu f(x):=\int_{\partial\Omega} \partial_{N,y}E_\mu(x,y)f(y)\rd S(y).$$

\begin{prop}
\label{singcontrest}
The operators 
$$\mathbbm{A}_\mu:=\gamma_{\partial\Omega}\circ \mathcal{A}_\mu, \quad \tilde{\mathbbm{B}}_\mu:=\partial_N\circ \mathcal{A}_\mu-1/2\quad \mbox{as well as}\quad \mathbbm{B}_\mu:=\gamma_{\partial\Omega}\circ \mathcal{B}_\mu-1/2$$
define elliptic pseudodifferential operators of order $-1$ on $\partial \Omega$ depending holomorphically on $\mu$.
\end{prop}

\begin{remark}
In fact, $\mathcal{B}_\mu f$ has a natural definition on $\R^{2d}\setminus \partial\Omega$ and makes a jump of size $f$ when crossing $\partial\Omega$, see more in for instance \cite[Chapter 3, Section 12]{agrruss}. This is the reason for subtracting $1/2$ from the exterior limit.
\end{remark}

\begin{proof}
The operator $\gamma_{\partial\Omega}\circ \mathcal{A}_\mu$ is a pseudo-differential operator of order $-1$ on $\partial\Omega$ by Lemma  \ref{imustruc}. The operators $\partial_N\circ \mathcal{A}_\mu$ and $\gamma_{\partial\Omega}\circ \mathcal{B}_\mu$ are pseudo-differential operators by a similar argument after verifying that $\partial_{N,y} I\left(\frac{\mu}{b},\frac{|x-y|^2}{2}\right)$ still admits a polyhomogeneous expansion starting at order $2-2d$ on $\partial \Omega$. Ellipticity of the involved operators follows from Lemma \ref{imustruc} which implies that the principal symbols are constant functions on $S^*\partial\Omega$. Holomorphicity is another immediate consequence of Lemma \ref{imustruc}.
\end{proof}

\begin{lem}
\label{singcont}
The single and double layer potential extend to bounded operators 
$$\mathcal{A}_\mu:H^{1/2}(\partial\Omega)\to H^2_A(\Omega)\quad\mbox{and}\quad \mathcal{B}_\mu:H^{3/2}(\partial\Omega)\to H^2_A(\Omega).$$ 
\end{lem}

\begin{proof}
We can find disjoint domains $\Omega',\Omega''\subseteq \Omega$ with $C^\infty$-boundaries, with $\Omega'$ being pre-compact, satisfying $\partial\Omega'=\partial\Omega''\dot{\cup} \partial\Omega$ and $\Omega=\Omega'\dot{\cup}\partial\Omega''\dot{\cup} \Omega''$. Let $r=\mathrm{dist}(\partial\Omega,\partial \Omega'')>0$. It follows from Lemma \ref{imustruc} that for some $N\in \Z$ and some constant $C>0$
$$\|\mathcal{A}_\mu f\|_{H^2_A(\Omega'')} \leq Cr^N\e^{-r} \|f\|_{L^2(\partial\Omega)}.$$
Similarly, $\|\mathcal{B}_\mu f\|_{H^2_A(\Omega'')} \leq Cr^N\e^{-r} \|f\|_{L^2(\partial\Omega)}$.

By elliptic regularity on the pre-compact $\Omega'$, we deduce the estimates
\begin{align*}
\|\mathcal{A}_\mu f\|_{H^2_A(\Omega')} &\leq C_{\Omega'}(\|\mathbbm{A}_\mu f\|_{H^{3/2}(\partial\Omega)}+\|\gamma_{\partial\Omega'}\mathcal{A}_\mu f\|_{H^{3/2}(\partial\Omega'')}+\|\mathcal{A}_\mu f\|_{L^2(\Omega')})\\
&\leq C_{\Omega'}(\|\mathbbm{A}_\mu\|_{H^{1/2}\to H^{3/2}}\| f\|_{H^{1/2}(\partial\Omega)}+\tilde{C}r^N\e^{-r} \|f\|_{L^2(\partial\Omega)}+\|\mathcal{A}_\mu f\|_{L^2(\Omega')}),\\
\|\mathcal{B}_\mu f\|_{H^2_A(\Omega')} &\leq C_{\Omega'}\left(\left\|\left(\frac{1}{2}+\mathbbm{B}_\mu\right)f \right\|_{H^{3/2}(\partial\Omega)}+\|\gamma_{\partial\Omega'}\mathcal{B}_\mu f\|_{H^{3/2}(\partial\Omega'')}+\|\mathcal{B}_\mu f\|_{L^2(\Omega')}\right)\\
&\leq C_{\Omega'}(\| f\|_{H^{3/2}(\partial\Omega)}+\tilde{C}r^N\e^{-r} \|f\|_{L^2(\partial\Omega)}+\|\mathcal{B}_\mu f\|_{L^2(\Omega')}).
\end{align*}
The second terms in both expressions,  $\|\gamma_{\partial\Omega'}\mathcal{A}_\mu f\|_{H^{3/2}(\partial\Omega'')}$ and $\|\gamma_{\partial\Omega'}\mathcal{B}_\mu f\|_{H^{3/2}(\partial\Omega'')}$, respectively, are estimated similarly as $\|\mathcal{A}_\mu f\|_{H^2_A(\Omega'')}$ was estimated above. Compactness of $\overline{\Omega'}$ and Lemma \ref{imustruc} guarantee that $\mathcal{A}_\mu,\mathcal{B}_\mu:L^2(\partial\Omega)\to L^2(\Omega')$ are bounded, hence $\|\mathcal{A}_\mu f\|_{L^2(\Omega')}\lesssim \|f\|_{L^2(\partial\Omega)}\lesssim  \|f\|_{H^{1/2}(\partial\Omega)}$ and similarly $\|\mathcal{B}_\mu f\|_{L^2(\Omega')}\lesssim  \|f\|_{H^{3/2}(\partial\Omega)}$. We conclude 
\begin{align*}
\|\mathcal{A}_\mu f\|_{H^2_A(\Omega)}&=\sqrt{\|\mathcal{A}_\mu f\|^2_{H^2_A(\Omega')}+\|\mathcal{A}_\mu f\|^2_{H^2_A(\Omega'')}}\lesssim \|f\|_{H^{1/2}(\partial\Omega)}\\
\|\mathcal{B}_\mu f\|_{H^2_A(\Omega)}&=\sqrt{\|\mathcal{B}_\mu f\|_{H^2_A(\Omega')}^2+\|\mathcal{B}_\mu f\|_{H^2_A(\Omega'')}^2}\lesssim \|f\|_{H^{3/2}(\partial\Omega)}.
\end{align*}
\end{proof}

\subsection{Dirichlet-to-Robin operators and similar constructions}

Green's formula implies that if $u\in H^2_A(\Omega)$ solves $(L_b-\mu)u=0$ in $\Omega$, then 
\begin{equation}
\label{greens}
u=\mathcal{B}_\mu(\gamma_{\partial\Omega}u)-\mathcal{A}_\mu(\partial_Nu).
\end{equation}
For details, see \cite[Chapter XX]{horIII}.

\begin{lem}
\label{muandlbdir}
If $\mu\in \C\setminus\sigma_L$, $\mu$ belongs to $\sigma(L^\Omega_{b,D})$ if and only if $\mathbbm{A}_\mu$ is non-invertible. Moreover, the single layer potential defines an isomorphism
\begin{equation}
\label{amudef}
\mathcal{A}_\mu|:\ker \mathbbm{A}_\mu\to \ker(L_{b,D}^\Omega-\mu).
\end{equation}
In particular, if $\mu\notin \sigma(L^\Omega_{b,D})$, the operator 
$$\mathcal{K}_{\mu,D}:=\mathcal{A}_\mu \circ \mathbbm{A}_\mu^{-1}:H^{3/2}(\partial\Omega)\to H^2_A(\Omega)$$
is continuous and for $f\in H^{3/2}(\partial\Omega)$, $u:=\mathcal{K}_{\mu,D}f$ is the unique solution to the boundary value problem
\begin{equation}
\label{direqho}
\begin{cases}
(L_b-\mu)u&=0,\quad\mbox{in}\quad \Omega,\\
\gamma_{\partial\Omega}u&=f\quad\mbox{on}\quad \partial\Omega.
\end{cases}
\end{equation}
\end{lem}

\begin{proof}
It follows from Lemma \ref{imustruc} that the principal symbol of $\mathbbm{A}_\mu$ is a constant function on $S^*\partial\Omega$, hence the Fredholm index of $\mathbbm{A}_\mu$ as an operator acting $H^s(\partial\Omega)\to H^{s+1}(\partial \Omega)$ vanishes for any $s$. As such, invertibility of $\mathbbm{A}_\mu$ is equivalent to $\ker \mathbbm{A}_\mu=0$. We claim that the mapping of Equation \eqref{amudef} is not only well defined but an isomorphism with inverse mapping defined from $-\partial_N|:\ker(L_{b,D}^\Omega-\mu)\to \ker \mathbbm{A}_\mu$. This follows from the fact that whenever $f\in \ker \mathbbm{A}_\mu\subseteq C^\infty(\partial\Omega)$, $u:=\mathcal{A}_\mu f\in H^2_A(\Omega)$ solves $(L_{b}-\mu)u=0$ in $\Omega$ and $\gamma_{\partial\Omega}u=\mathbbm{A}_\mu f=0$; we conclude that $u\in \Dom(L_{b,D}^\Omega)$ and $(L_{b,D}^\Omega-\mu)u=0$. Conversely, if $u\in \ker(L_{b,D}^\Omega-\mu)$ then by Green's formula \eqref{greens}, $u=-\mathcal{A}_\mu(\partial_N u)$ and since $u\in \Dom(L_{b,D}^\Omega)$, $\mathbbm{A}_\mu(-\partial_N u)=\gamma_{\partial\Omega} u=0$.

By the argument above, $\mathbbm{A}_\mu^{-1}$ is a well defined pseudo-differential operator of order $1$ if $\mu\notin \sigma(L^\Omega_{b,D})$. Hence, the operator $\mathcal{K}_{\mu,D}$ is indeed continuous by Lemma \ref{singcont} with $u:=\mathcal{K}_{\mu,D}f$ satisfying \eqref{direqho}. Uniqueness of the solution to \eqref{direqho} follows from $\mu\notin\sigma(L_{b,D}^\Omega)$.
\end{proof}

We will make use of the shortened notation $\sigma_D:=\sigma(L^\Omega_{b,D})$.

\begin{lem}
\label{muandlbtau}
Let $\mu\in \C\setminus \sigma_D$ and $\tau\in \Psi^0(\partial\Omega)^{s.a.}$. The number $\mu$ belongs to $\sigma(L_{b,\tau}^\Omega)$ if and only if $-1/2+\mathbbm{B}_\mu+\mathbbm{A}_\mu \tau$ is non-invertible. Moreover, there is an isomorphism
\begin{equation}
\label{amubmudef}
(\mathcal{B}_\mu+\mathcal{A}_\mu\circ \tau)|:\ker (-1/2+\mathbbm{B}_\mu+\mathbbm{A}_\mu \tau)\to \ker(L_{b,\tau}^\Omega-\mu).
\end{equation}
In particular, if $\mu\notin \sigma(L^\Omega_{b,\tau})\cup \sigma_D$, the operator 
$$\mathcal{K}_{\mu,\tau}:=\mathcal{A}_\mu \circ \mathbbm{A}_\mu^{-1} (-1/2+\mathbbm{B}_\mu+\mathbbm{A}_\mu\tau)^{-1}\mathbbm{A}_\mu:H^{1/2}(\partial\Omega)\to H^2_A(\Omega)$$
is continuous and for $f\in H^{1/2}(\partial\Omega)$, $u:=\mathcal{K}_{\mu,\tau}f$ is the unique solution to the boundary value problem
\begin{equation}
\label{direqhorob}
\begin{cases}
(L_b-\mu)u&=0,\quad\mbox{in}\quad \Omega,\\
(\partial_N+\tau\gamma_{\partial\Omega})u&=f\quad\mbox{on}\quad \partial\Omega.
\end{cases}
\end{equation}
\end{lem}

\begin{proof}
The zero order operator $-1/2+\mathbbm{B}_\mu+\mathbbm{A}_\mu \tau$ is elliptic with constant symbol $-1/2$, because $\mathbbm{A}_\mu\tau$ and $\mathbbm{B}_\mu$ are of order $-1$. Hence invertiblity of $-1/2+\mathbbm{B}_\mu+\mathbbm{A}_\mu \tau$ is equivalent to injectivity. We claim that for $\mu\notin \sigma_D$ not only is the mapping \eqref{amubmudef} well defined, but it is an isomorphism with inverse being the trace operator 
\begin{equation}
\label{gammadefinv}
\gamma_{\partial\Omega}|: \ker(L_{b,\tau}^\Omega-\mu)\to \ker (-1/2+\mathbbm{B}_\mu+\mathbbm{A}_\mu \tau).
\end{equation}
In fact, for any $\mu \notin \sigma_L$, if $u\in \ker(L_{b,\tau}^\Omega-\mu)$ Green's formula implies $u=(\mathcal{B}_\mu+\mathcal{A}_\mu \tau) \gamma_{\partial\Omega}u$ hence $\gamma_{\partial\Omega}u\in \ker (-1/2+\mathbbm{B}_\mu+\mathbbm{A}_\mu \tau)$. Hence the mapping in \eqref{gammadefinv} is always well defined. To prove that the mapping in \eqref{amubmudef} is well defined, we note that for $f\in \ker (-1/2+\mathbbm{B}_\mu+\mathbbm{A}_\mu \tau)\subseteq H^{3/2}(\partial\Omega)$, Green's formula applied to $u= (\mathcal{B}_\mu+\mathcal{A}_\mu\circ \tau)f$ implies that $\mathcal{A}_\mu(\partial_N+\tau\gamma_{\partial\Omega})u=0$ hence $\mathbbm{A}_\mu(\partial_N+\tau\gamma_{\partial\Omega})u=0$. If $\mu\notin \sigma_D$, Lemma \ref{muandlbdir} implies that $(\partial_N+\tau\gamma_{\partial\Omega})u=0$ hence $u\in \Dom(L^\Omega_{b,\tau})$ and $u\in \ker(L^\Omega_{b\tau}-\mu)$ follows from the construction. The verification that the mappings in \eqref{amubmudef} and \eqref{gammadefinv} are each others inverses follows from a simple exercise in linear algebra.

To prove that $\mathcal{K}_{\mu,\tau}:=\mathcal{A}_\mu \circ \mathbbm{A}_\mu^{-1} (-1/2+\mathbbm{B}_\mu+\mathbbm{A}_\mu\tau)^{-1}\mathbbm{A}_\mu$ is the solution operator to \eqref{direqhorob} for $\mu\notin \sigma(L^\Omega_{b,\tau})\cup \sigma_D$, we make the ansatz $u=\mathcal{A}_\mu g$ for some $g\in H^{1/2}(\partial\Omega)$. Green's formula implies  
$$\mathbbm{A}_\mu f=(-1/2+\mathbbm{B}_\mu+\mathbbm{A}_\mu\tau)\mathbbm{A}_\mu g.$$
It follows that $g= \mathbbm{A}_\mu^{-1} (-1/2+\mathbbm{B}_\mu+\mathbbm{A}_\mu\tau)^{-1}\mathbbm{A}_\mu f$ and $u=\mathcal{K}_{\mu,\tau} f$. Uniqueness of the solution to \eqref{direqhorob} follows because $\mu\notin \sigma(L^\Omega_{b,\tau})$.
\end{proof}

\begin{deef}
\label{drop}
Let $\tau, \tau'\in \Psi^0(\partial\Omega)^{s.a}$ and $\mu\in \C$. We define the following operators on $C^\infty(\partial\Omega)$: 
\begin{align*}
\Lambda^{D\to R}_\tau(\mu)&:=(\partial_N+\tau\gamma_{\partial\Omega})\circ\mathcal{K}_{\mu,D},\;&\mbox{for}\;\mu\notin\sigma_D. \quad&\mbox{(\emph{Dirichlet-to-Robin operator})}\\
\Lambda^{R\to D}_\tau(\mu)&:=\gamma_{\partial\Omega}\circ\mathcal{K}_{\mu,\tau},\;&\mbox{for}\;\mu\notin\sigma(L^\Omega_{b,\tau})\cup\sigma_D. \quad&\mbox{(\emph{Robin-to-Dirichlet operator})}\\
\Lambda^{R\to R}_{\tau\to \tau'}(\mu)&:=(\partial_N+\tau'\gamma_{\partial\Omega})\circ\mathcal{K}_{\mu,\tau},\;&\mbox{for}\;\mu\notin\sigma(L^\Omega_{b,\tau})\cup\sigma_D. \quad&\mbox{(\emph{Robin-to-Robin operator})}
\end{align*}
\end{deef}

\begin{prop}
\label{dtorprop}
Let $\mu \in \C\setminus \sigma_D$ and $\tau, \tau'\in \Psi^0(\partial\Omega)^{s.a}$. The operators of Definition \ref{drop} possess the following properties:
\begin{enumerate}
\item Whenever the operators 
$$\Lambda^{D\to R}_\tau(\mu)\in \Psi^1(\partial\Omega),\quad \Lambda^{R\to D}_\tau(\mu)\in\Psi^{-1}(\partial\Omega)\quad\mbox{and}\quad \Lambda^{R\to R}_{\tau\to \tau'}(\mu)\in \Psi^0(\partial\Omega)$$ 
are defined, they are elliptic with constant principal symbol, self-adjoint and bounded from below.
\item Whenever the expressions make sense, 
\begin{align*}
\Lambda^{D\to R}_\tau(\mu)=\Lambda^{R\to D}_\tau(\mu)^{-1}, \quad &\Lambda^{R\to R}_{\tau\to \tau'}(\mu)=\Lambda^{R\to R}_{\tau'\to \tau}(\mu)^{-1}\\
\mbox{and}\quad &\Lambda^{D\to R}_{\tau'}(\mu)\circ \Lambda^{R\to D}_\tau(\mu)=\Lambda^{R\to R}_{\tau\to \tau'}(\mu).
\end{align*}
\item In terms of the pseudo-differential operators $\mathbbm{A}_\mu$, $\mathbbm{B}_\mu$ and $\tau$,
\begin{align*}
\Lambda^{D\to R}_\tau(\mu)=&\mathbbm{A}_\mu^{-1}(-1/2+\mathbbm{B}_\mu+\mathbbm{A}_\mu\tau)\\
\mbox{and}\quad &\Lambda^{R\to R}_{\tau\to \tau'}(\mu)=\mathbbm{A}_\mu^{-1}(-1/2+\mathbbm{B}_\mu+\mathbbm{A}_\mu\tau')(-1/2+\mathbbm{B}_\mu+\mathbbm{A}_\mu\tau)^{-1}\mathbbm{A}_\mu,
\end{align*}
whenever the expressions make sense.
\end{enumerate}
\end{prop}

With Lemma \ref{muandlbdir} and \ref{muandlbtau} at hand, Proposition \ref{dtorprop} follows from standard techniques and we refrain from proving it here. The reader can find details in \cite[Appendix C of Chapter 12]{tay2}. The operators of Proposition \ref{dtorprop} can be defined modulo finite rank smoothing operators for any $\mu \in \C\setminus\sigma_L$.

\begin{remark}
\label{dirichlrobrem}
Lemma \ref{muandlbtau} and Proposition \ref{dtorprop} prove the first part of Theorem \ref{localizintobo} (on page \pageref{localizintobo}). To reconcile with the notation of Theorem \ref{localizintobo}, we define
$$\Gamma(\mu,\tau):=-2(\mathbbm{B}_\mu+\mathbbm{A}_\mu \tau).$$
\end{remark}

From the results of this subsection we conclude the following Theorem that forms the main technical ingredient needed to prove gap continuity in the next section.

\begin{thm}
\label{constructingsolop}
Let $\tau\in \Psi^0(\partial\Omega)^{s.a.}$. For any $\mu\in \C\setminus (\sigma(L^\Omega_{b,\tau})\cup \sigma_D)$ the boundary value problem
\begin{equation}
\label{fullbvp}
\begin{cases}
(L_b-\mu)u&=u_0,\quad\mbox{in}\quad \Omega,\\
(\partial_N+\tau\gamma_{\partial\Omega})u&=f,\quad\mbox{on}\quad \partial\Omega,
\end{cases}
\end{equation}
admits a unique solution $u\in H^2_A(\Omega)$ for $u_0\in L^2(\Omega)$ and $f\in H^{1/2}(\partial\Omega)$. Furthermore, letting $R_{\mu,\tau}:L^2(\Omega)\to H^2_A(\Omega)$ denote the inverse of $L^\Omega_{b,\tau}-\mu$ and $\mathcal{K}_{\mu,\tau}$ the operator of Lemma \ref{muandlbtau}, the solution operator to \eqref{fullbvp} takes the form
$$\begin{pmatrix}
R_{\mu,\tau} & \mathcal{K}_{\mu,\tau}
\end{pmatrix}
:\;\begin{matrix} L^2(\Omega)\\\oplus\\ H^{1/2}(\partial\Omega)\end{matrix}\longrightarrow\; H^2_A(\Omega),$$
and depends holomorphically on $\mu\in \C\setminus  (\sigma(L^\Omega_{b,\tau})\cup \sigma_D)$.
\end{thm}

\begin{remark}
The analog of Theorem \ref{constructingsolop} for Dirichlet conditions of course also holds for any $\mu\in \C\setminus \sigma_D$.
\end{remark}

\large
\section{Gap continuity and spectral flows}
\normalsize

In this section we will briefly recall some notions and results on spectral flow. Our main reference for these results is \cite{bobaleph}. We will proceed by proving that the Landau-Robin hamiltonians parametrized by their Robin data satisfy the necessary continuity condition from \cite{bobaleph} for defining their spectral flow.

\subsection{Spectral flow and gap continuity} 
\label{sfsubsection}
We will use $\He$ to denote a separable Hilbert space, e.g. $L^2(\Omega)$. Recall the notation $\mathcal{CF}^{s.a.}(\He)$ for the set of closed Fredholm operators defined in $\He$ with the additional property of being self-adjoint. The topology of $\mathcal{CF}^{s.a.}(\He)$ that behaves well with spectral flow is the gap-topology; it is defined from the metric
$$d_{gap}(T_1,T_2):=\|(T_1+i)^{-1}-(T_2+i)^{-1}\|_{\Bo(\He)}.$$
In particular, for a topological space $X$, a function $f:X\to \mathcal{CF}^{s.a.}(\He)$ is continuous in the gap topology if and only if $(f+i)^{-1}:X\to \Bo(\He)$ is continuous in norm topology. The set of invertible elements in $\mathcal{CF}^{s.a.}(\He)$ is open in the gap topology by \cite[Proposition 1.7]{bobaleph}. By \cite[Section 1.1]{bobaleph}, the metric $d_{gap}$ is equivalent to the metric defined from the norm distance between the graph projections and also to the metric defined from norm distance for the Cayley transform
$$\kappa:\mathcal{CF}^{s.a.}(\He)\to \mathcal{U}(\He), \quad T\mapsto (T-i)(T+i)^{-1}.$$ 
The image of the Cayley transform is characterized in \cite[Theorem 1.10]{bobaleph} as the set of unitaries $U\in \mathcal{U}(\He)$ such that $1+U$ is Fredholm and $1-U$ is injective. Remarkably, by \cite[Proposition 1.6]{bobaleph}, the subspace $\Fg^{s.a.}(\He)\subseteq \mathcal{CF}^{s.a.}(\He)$ of bounded self-adjoint Fredholm operators is dense in the gap topology. Another surprising property is that while the subspace $\Fg^{s.a.}(\He)$ has three path-components, $\mathcal{CF}^{s.a.}(\He)$ is path-connected by \cite[Theorem 1.10]{bobaleph}.

Let us briefly recall a construction of the spectral flow of a gap continuous path $f:[0,1]\to \mathcal{CF}^{s.a.}(\He)$ from \cite{bobaleph}. Following \cite{kirklesch}, there is a winding number construction 
$$\wind:C([0,1], \,_\Fg\!\mathcal{U}(\He))\to \Z,$$
where $_\Fg\!\mathcal{U}(\He)$ is the set of all unitaries $U$ such that $1+U$ is Fredholm -- a space containing $\kappa(\mathcal{CF}^{s.a.}(\He))$ as a dense subset. One defines 
$$\mathrm{sf}(f)=\wind(\kappa\circ f).$$
Two main properties of the spectral flow are its additivity and homotopy invariance:
\begin{enumerate}
\item (\emph{Additivity}) If $f:[0,2]\to  \mathcal{CF}^{s.a.}(\He)$ is gap continuous, 
$$\mathrm{sf}(f)=\mathrm{sf}(f|_{[0,1]})+\mathrm{sf}(f|_{[1,2]}).$$
\item (\emph{Homotopy invariance}) If $F:[0,1]\times [0,1]\to  \mathcal{CF}^{s.a.}(\He)$ is gap continuous and $\dim\ker F(0,s)$ and $\dim \ker F(1,s)$ are constant, then $\mathrm{sf}(F(\cdot,s))$ is independent of $s\in [0,1]$.
\end{enumerate} 
We refer the reader to \cite[Section 2.1]{bobaleph} for proofs of these two properties. To compute spectral flows, we will as a rule use the next proposition.

\begin{prop}[Proposition 2.1 of \cite{bobaleph}, cf. Proposition 2.17 of \cite{bobaleph}]
\label{computationaflwo}
Given a gap continuous path $f:[0,1]\to \mathcal{CF}^{s.a.}(\He)$, there is a partition $0=t_0<t_1<\cdots <t_n=1$ and $\lambda_j>0$ for $j=1,...,n$ such that $\ker(f(t)-\lambda_j)=0$ for all $t\in [t_{j-1},t_j]$ and 
$$\mathrm{sf}(f)=\sum_{j=1}^n\sum_{\lambda\in [0,\lambda_j)}\dim \ker(f(t_j)-\lambda)-\dim \ker(f(t_{j-1})-\lambda).$$

\end{prop}

\subsection{Gap continuity of $L^\Omega_{b,\tau}$ and Theorem \ref{gapconthm}}
\label{provingmainthm}
Motivated by the recollection of results in the previous subsection, we now turn to proving Theorem \ref{gapconthm} (see page \pageref{gapconthm}). It follows directly from the following lemma.

\begin{lem}
\label{gapestimates}
Let $\tau_0,\tau_1\in \Psi^0(\partial\Omega)^{s.a.}$ and $\lambda\in \R$. The resolvent difference of Robin-Landau operators is given by
$$(L^\Omega_{b,\tau_1}-\lambda-i)^{-1}-(L^\Omega_{b,\tau_0}-\lambda-i)^{-1}=-\mathcal{K}_{\lambda+i,\tau_0}\Lambda_{\tau_1\to \tau_0}^{R\to R}(\lambda+i)(\tau_1-\tau_0)\gamma_{\partial\Omega}(L^\Omega_{b,\tau_0}-\lambda-i)^{-1},$$
where $\mathcal{K}_{\lambda+i,\tau_0}$ is the Poisson operator of Lemma \ref{muandlbtau} at $\mu=\lambda+i$ and $\Lambda_{\tau_1\to \tau_0}^{R\to R}(\lambda+i)$ is the Robin-to-Robin operator (see Definition \ref{drop}) at $\mu=\lambda+i$. In particular, 
$$\|(L^\Omega_{b,\tau_1}-\lambda-i)^{-1}-(L^\Omega_{b,\tau_0}-\lambda-i)^{-1}\|_{\Bo(L^2(\Omega))}\leq C(\lambda,\tau_0,\tau_1,\Omega) \|\tau_1-\tau_0\|_{\Bo(H^{3/2}(\partial\Omega),H^{1/2}(\partial\Omega))},$$
where $C(\lambda,\tau_0,\tau_1,\Omega)$ is the locally bounded number
\small
$$\|\mathcal{K}_{\lambda+i,\tau_0}\|_{\Bo(H^{1/2}(\partial\Omega), H^2_A(\Omega))}\|\Lambda_{\tau_1\to \tau_0}^{R\to R}(\lambda+i)\|_{\Bo(H^{1/2}(\partial\Omega))}\|\gamma_{\partial\Omega}\|_{\Bo(H^2_A(\Omega), H^{3/2}(\partial\Omega))}\|(L^\Omega_{b,\tau_0}-\lambda-i)^{-1}\|_{\Bo(L^2(\Omega), H^2_A(\Omega))}.$$
\normalsize
In particular, for $\lambda\in \R\setminus \sigma_L$ the mapping 
$$\Psi^0(\partial\Omega)^{s.a.}\ni \tau \mapsto L^\Omega_{b,\tau}-\lambda\in  \mathcal{CF}^{s.a.}(L^2(\Omega)),$$
is well defined and gap continuous when equipping $\Psi^0(\partial\Omega)^{s.a.}$ with the topology induced from the norm in $\Bo(H^{3/2}(\partial\Omega),H^{1/2}(\partial\Omega))$.
\end{lem}

\begin{proof}
The operators occurring in the theorem are all well defined by Lemma \ref{muandlbtau} and Theorem \ref{constructingsolop} because we are considering $\mu=\lambda+i$. In the notation of Theorem \ref{constructingsolop}, we have the following identity of operators on $L^2(\Omega)\oplus H^{1/2}(\partial\Omega)$
\begin{align*}
\begin{pmatrix}
L_b-\lambda-i\\
\partial_N+\tau_1\gamma_{\partial\Omega} 
\end{pmatrix}
\begin{pmatrix}
R_{\lambda+i,\tau_0} & \mathcal{K}_{\lambda+i,\tau_0}
\end{pmatrix}&= 
\begin{pmatrix}
1& 0\\
(\partial_N+\tau_1\gamma_{\partial\Omega}) R_{\lambda+i,\tau_0}& (\partial_N+\tau_1\gamma_{\partial\Omega} )\mathcal{K}_{\lambda+i,\tau_0}
\end{pmatrix}\\
&=\begin{pmatrix}
1& 0\\
(\tau_1-\tau_0)\gamma_{\partial\Omega} R_{\lambda+i,\tau_0}& \Lambda_{\tau_0\to \tau_1}^{R\to R}(\lambda+i)
\end{pmatrix}.
\end{align*} 
It follows from these computations and Proposition \ref{dtorprop} that
\begin{align*}
&\begin{pmatrix}
R_{\lambda+i,\tau_1} & \mathcal{K}_{\lambda+i,\tau_1}
\end{pmatrix}\\
&\qquad=\begin{pmatrix}
R_{\lambda+i,\tau_0} & \mathcal{K}_{\lambda+i,\tau_0}
\end{pmatrix}\cdot \begin{pmatrix}
1& 0\\
-\Lambda_{\tau_1\to \tau_0}^{R\to R}(\lambda+i)(\tau_1-\tau_0)\gamma_{\partial\Omega} R_{\lambda+i,\tau_0}& \Lambda_{\tau_1\to \tau_0}^{R\to R}(\lambda+i)
\end{pmatrix}\\
\\
&\qquad=\begin{pmatrix}
R_{\lambda+i,\tau_0}-\mathcal{K}_{\lambda+i,\tau_0}\Lambda_{\tau_1\to \tau_0}^{R\to R}(\lambda+i)(\tau_1-\tau_0)\gamma_{\partial\Omega} R_{\lambda+i,\tau_0} & \mathcal{K}_{\lambda+i,\tau_0}\Lambda_{\tau_1\to \tau_0}^{R\to R}(\lambda+i)
\end{pmatrix}.
\end{align*}
\end{proof}

\begin{remark}
The norm estimate on $\|(L^\Omega_{b,\tau_1}-i)^{-1}-(L^\Omega_{b,\tau_0}-i)^{-1}\|_{\Bo(L^2(\Omega))}$ in Lemma \ref{gapestimates} still holds true for $\tau_0,\tau_1\in \Psi^t(\partial\Omega)^{s.a.}$ for any $t<1$.
\end{remark}

\begin{remark}
By similar computations as in Lemma \ref{gapestimates}, for $\mu\notin \sigma(L^\Omega_{b,\tau_1})\cup \sigma(L^\Omega_{b,\tau_2})$, the operator $(L^\Omega_{b,\tau_1}-\mu)^{-1}-(L^\Omega_{b,\tau_0}-\mu)^{-1}$ factors over the inclusion $H^{3/2}(\partial\Omega)\hookrightarrow H^{1/2}(\partial\Omega)$. This recovers the wellknown result, in the style of Birman \cite{birrre61b}, that $(L^\Omega_{b,\tau_1}-\mu)^{-1}-(L^\Omega_{b,\tau_0}-\mu)^{-1}\in \mathcal{L}^{2d-1,\infty}(L^2(\Omega))$ -- the weak Schatten class of exponent $2d-1$. The asymptotics of the singular numbers of the resolvent difference was computed in \cite[Theorem 3.4]{Grubbeligrubb}.
\end{remark}

\begin{remark}
Corollary \ref{sfcor} (on page \pageref{sfcor}) follows directly from Theorem \ref{gapconthm} and homotopy invariance of the spectral flow (see \cite[Proposition 2.3]{bobaleph}), using that $\Psi^0(\partial\Omega)^{s.a.}$ is a linear space, hence contractible in any topology defined from a semi-norm. That $\mathrm{sf}\,(L^\Omega_{b,\tau_t}-\mu)_{t\in [0,1]}$ only depends on a neighborhood of $\tau_1$ in the $\Bo(H^{3/2}(\partial\Omega),H^{1/2}(\partial\Omega))$-topology on $\Psi^0(\partial\Omega)^{s.a.}$, assuming $\mu\notin\sigma(L^\Omega_{b,\tau_1})$, follows from the fact that the set of invertible elements in $\mathcal{CF}^{s.a.}(L^2(\partial\Omega))$ is open (see \cite[Proposition 1.7]{bobaleph}) and the homotopy invariance of the spectral flow (see \cite[Proposition 2.3]{bobaleph}).
\end{remark}

\begin{prop}
\label{proooooftauapr}
Recall the notation of Remark \ref{smoothingappsremk}, on page \pageref{smoothingappsremk}. For $\tau\in \Psi^t(\partial \Omega)$ the following estimate holds 
\begin{align*}
\|\tau-T_{(N)}(\tau)\|&_{\Bo(H^{3/2}(\partial\Omega),H^{1/2}(\partial\Omega))}\\
&\leq C\left(\|\tau\|_{\Bo(H^{3/2}(\partial\Omega),H^{3/2-t}(\partial\Omega))}+\|\tau\|_{\Bo(H^{1/2+t}(\partial\Omega),H^{1/2}(\partial\Omega))}\right) N^{-\frac{1-t}{2d-1}}.
\end{align*}
\end{prop}

\begin{proof}
We can without loss of generality assume $D$ to be of order $1$ with $De_k=k^{\frac{1}{2d-1}}e_k$. Take an $f\in H^{3/2}(\partial\Omega)$ and write $f=\sum_{k=1}^\infty f_ke_k$ with $f_k=\langle f,e_k\rangle_{L^2(\partial\Omega)}$. For any $g\in H^{s}(\partial\Omega)$, $(k^{\frac{s}{2}}g_k)_{k\in \N_+}\in \ell^2(\N_+)$ and we can in fact assume that $\|g\|_{H^s(\partial\Omega)}=\|(k^{\frac{s}{2}}g_k)_{k\in \N_+}\|_{\ell^2(\N_+)}$. We write 
$$\left[\tau-T_{(N)}(\tau)\right]f=\sum_{\max(j,k)>N} \langle \tau e_k,e_j\rangle_{L^2(\partial\Omega)}\langle f,e_k\rangle_{L^2(\partial\Omega)}e_j$$
Define the pseudodifferential projection $P_N:=\sum_{k=N+1}^\infty e_k\otimes e_k^*\in \Psi^0(\partial\Omega)$. It follows that 
\begin{align*}
\bigg\|\bigg[\tau-&T_{(N)}(\tau)\bigg]f\bigg\|_{H^{1/2}(\partial\Omega)}^2=\sum_{j=1}^N\left|\sum_{k=N+1}^\infty j^{\frac{1}{4d-2}}\langle \tau e_k,e_j\rangle_{L^2(\partial\Omega)}\langle f,e_k\rangle_{L^2(\partial\Omega)}\right|^2\\
&\qquad\qquad\qquad\qquad\qquad\qquad+\sum_{j=N+1}^\infty\left|\sum_{k=1}^\infty j^{\frac{1}{4d-2}}\langle \tau e_k,e_j\rangle_{L^2(\partial\Omega)}\langle f,e_k\rangle_{L^2(\partial\Omega)}\right|^2\\
&\leq N^{-2\frac{1-t}{2d-1}}\bigg(\sum_{j=1}^N\left|\sum_{k=N+1}^\infty j^{\frac{1}{4d-2}}\langle \tau e_k,e_j\rangle_{L^2(\partial\Omega)}\langle D^{1-t}f,e_k\rangle_{L^2(\partial\Omega)}\right|^2\\
&\qquad\qquad\qquad\qquad+\sum_{j=N+1}^\infty\left|\sum_{k=1}^\infty j^{\frac{3-2t}{4d-2}}\langle \tau e_k,e_j\rangle_{L^2(\partial\Omega)}\langle f,e_k\rangle_{L^2(\partial\Omega)}\right|^2\bigg)\\
&\leq N^{-2\frac{1-t}{2d-1}}\bigg(\sum_{j=1}^\infty\left|\sum_{k=N+1}^\infty j^{\frac{1}{4d-2}}\langle \tau e_k,e_j\rangle_{L^2(\partial\Omega)}\langle D^{1-t}f,e_k\rangle_{L^2(\partial\Omega)}\right|^2\\
&\qquad\qquad\qquad\qquad+\sum_{j=1}^\infty\left|\sum_{k=1}^\infty j^{\frac{3-2t}{4d-2}}\langle \tau e_k,e_j\rangle_{L^2(\partial\Omega)}\langle f,e_k\rangle_{L^2(\partial\Omega)}\right|^2\bigg)\\
&\leq N^{-2\frac{1-t}{2d-1}}\bigg(\left\|\tau D^{1-t}P_Nf\right\|^2_{H^{1/2}(\partial\Omega)}+\left\|\tau f\right\|^2_{H^{3/2-t}(\partial\Omega)}\bigg)\\
&\leq N^{-2\frac{1-t}{2d-1}}\left(\|\tau\|_{\Bo(H^{1/2+t}(\partial\Omega),H^{1/2}(\partial\Omega))}^2 +\|\tau\|_{\Bo(H^{3/2}(\partial\Omega),H^{3/2-t}(\partial\Omega))}^2\right)\left\|f\right\|_{H^{3/2}(\partial\Omega)}^2.
\end{align*}

\end{proof}

\subsection{Holomorphic families and reduction to the boundary}
\label{holoandloc}

A path $(\tau_t)_{t\in [0,1]}\subseteq \Psi^0(\partial\Omega)^{s.a.}$ is said to be holomorphic if it is the restriction of a holomorphic function $\tau:U\to  \Psi^0(\partial\Omega)$, where $\C\supseteq U\supseteq [0,1]$ is an open neighborhood. In this section we will give a direct proof of the fact that whenever $(\tau_t)_{t\in [0,1]}\subseteq \Psi^0(\partial\Omega)^{s.a.}$ is a holomorphic path, we can parametrize eigenvalues locally as functions with a holomorphic extension. By Lemma \ref{gapestimates}, the family $(L^\Omega_{b,\tau_t}-\mu_0)_{t\in U}$ is an analytic family in the sense of Kato for any $\mu_0$ (for the definition of this notion, see \cite[Chapter XII.2, Page 14]{reedsimoniv}). The following theorem describes the flow of specific eigenvalues as the Robin data vary. The theorem could also be deduced from \cite[Theorem XII.13]{reedsimoniv} using Lemma \ref{gapestimates}, see also \cite[Theorem VII.1.8]{katobook}.

\begin{thm}
\label{parametrizingmu}
Suppose that $\mu_0\in \R\setminus \sigma_D$ and $(\tau_t)_{t\in [0,1]}\subseteq \Psi^0(\partial\Omega)^{s.a.}$ is a holomorphic path. There is a partition $0=t_0<t_1<...<t_{M-1}<t_M=1$ such that, for $k=1,\ldots,M$, the interval $[t_{k-1},t_{k}]$ admits an open neighborhood $V_k$ in $\C$ on which there is a finite collection of bounded holomorphic functions $(\mu_{jk})_{j=1}^{N_k}\subseteq \mathcal{O}(V_k)$ and an open neighborhood $W_k\subseteq \C$ of $\mu_0$ such that (taking multiplicities into account)
$$\sigma(L^\Omega_{b,\tau_t})\cap W_k=\cup_{j=1}^{N_k} \mu_{jk}(t)\cap W_k \quad\forall t\in [t_{k-1},t_k].$$
Moreover, there is a collection $(u_{jk})_{j=1}^{N_k}\subseteq \mathcal{O}(V_k,H^2_A(\Omega))$ such that for $t\in [t_{k-1},t_k]$,
$$u_{jk}(t)\in \Dom(L^\Omega_{b,\tau_t})\setminus \{0\}\quad\mbox{and}\quad L^\Omega_{b,\tau_t}u_{jk}(t)=\mu_{jk}(t)u_{jk}(t).$$
\end{thm}

\begin{proof}
We start by proving the first part of the theorem concerning the parametrization of the spectrum. Let $W_0$ be an open neighborhood of $\mu_0$ such that $W_0$ does not intersect $\sigma_D$. By Lemma \ref{muandlbtau}, it holds for any $t\in [0,1]$ that
$$\sigma(L^\Omega_{b,\tau_t})\cap W_0=\{\mu\in W_0: \; 1-2\mathbbm{B}_\mu-2\mathbbm{A}_\mu \tau_t \quad\mbox{non-invertible}\}.$$
Since $\mathbbm{B}_\mu, \mathbbm{A}_\mu\in \Psi^{-1}(\partial\Omega)$ depend holomorphically on $\mu$, we can define the holomorphic function
$$\pmb{f}\in \mathcal{O}(U\times W_0), \quad \pmb{f}(t,\mu):=\det\,\!_{2d}(1-2\mathbbm{B}_\mu-2\mathbbm{A}_\mu \tau(t) ).$$
Here $\det_{2d}$ denotes the regularized determinant (see \cite[Chapter 9]{simon}). The regularized determinant $\det_{2d}$ defines  a holomorphic function $1+\mathcal{L}^{2d}(L^2(\partial\Omega))\to \C$ such that $\det_{2d}(1+K)=0$ if and only if $1+K$ is non-invertible. Hence, $\sigma(L^\Omega_{b,\tau_t})\cap W_0=\{\mu\in W_0: \, \pmb{f}(t,\mu)=0\}$. 

Fix a point $s_0\in [0,1]$. We can by \cite[Subsection 3.3 and 3.4]{range} find a holomorphic ${\pmb u}$ which non-zero in a neighborhood of $(s_0,\mu_0)$ and irreducible holomorphic functions $\pmb{f}_ {1},\ldots,\pmb{f}_{m}$ such that near $(s_0,\mu_0)$, 
$$\pmb{f}(t,\mu)={\pmb u}(t,\mu)\pmb{f}_{1}(t,\mu)\pmb{f}_{2}(t,\mu)\cdots \pmb{f}_{m}(t,\mu).$$
We will construct the functions $\mu_{j}$, for $j=1,...,N$, as the holomorphic parametrizations $\mu=\mu(t)$ of the analytic sets $\pmb{f}_k(t,\mu)=0$ near $(s_0,\mu_0)$ as $k$ ranges from $1$ to $m$. Fix a $k$ and consider a function $\pmb{f}_k$ as above. Near $(s_0,\mu_0)$ we can by Puiseux' theorem (see \cite[Theorem 2.2.6]{ctcwall}) parametrize the zero set $\pmb{f}_k(t,\mu)=0$ as a multi-valued function $\mu(t)=\mathfrak{m}(t-s_0)$ for a multivalued holomorphic function $\mathfrak{m}$ with $\mathfrak{m}'(z)\neq 0$ for $z$ near $0$ (in all its branches). Moreover, all branches of $\mathfrak{m}$ are holomorphic outside $s_0$ and holomorphic near $s_0$ after a suitable singular coordinate change. Following the proof of \cite[Theorem XII.3]{reedsimoniv}, we can by Puiseux' theorem Taylor expand the branch $\mu(t)$ at $t=s_0$ as $\mu(t)=\mu_0+\sum_{j=1}^\infty \beta_j (t-s_0)^{j/p}$, for some $p\in \N_{>0}$ and coefficients $\beta_j$. Since $\mu(t)$ is real whenever $t$ is real, both the numbers
$$\e^{\pi i/p}\beta_1=\lim_{t\uparrow s_0} \frac{\mu(t)-\mu_0}{(t-s_0)^{1/p}} \quad\mbox{and}\quad \beta_1=\lim_{t\downarrow s_0} \frac{\mu(t)-\mu_0}{(t-s_0)^{1/p}},$$
are real. We deduce that $\beta_1=0$. By induction one can show that $\beta_j=0$ unless $p|j$. Hence, $\mu(t)$ is holomorphic at $t=s_0$. This construction gives rise to potentially several parametrizations of $\pmb{f}_k(t,\mu)=0$ holomorphic near $s_0$. Since $[0,1]$ is compact, the first part of the theorem follows. 

Let us turn to the second part of the theorem, concerning the eigenfunctions. To simplify notation, we drop the $k$ in the notation throughout the rest of the proof. To prove existence of  $(u_j)_{j=1}^N\subseteq \mathcal{O}(V,H^2_A(\Omega))$, we will for simplicity reduce the problem to the boundary; we need to prove existence of $(g_j)\subseteq  \mathcal{O}(V,H^{3/2}(\partial\Omega))$, for some open neighborhood $V$ of $s_0\in \cup_{l=1}^{N_0}(w_{l-1},w_l)$, such that $g_j(t)\in \ker(-1/2+\mathbbm{B}_{\mu_j(t)}+\mathbbm{A}_{\mu_j(t)}\tau_t)\setminus \{0\}$ for all $t\in V\cap [0,1]$. For a small enough neighborhood $V$, Lemma \ref{muandlbtau} implies that the collection $(u_j)_{j=1}^N$ can be constructed from $(g_j)_{j=1}^N$ by means of the formula $u_j(t):=(\mathcal{B}_{\mu_j(t)}+\mathcal{A}_{\mu_j(t)}\tau_t)g_j(t)$. It follows from Lemma \ref{muandlbtau} and the construction of $(\mu_j)_{j=1}^N\subseteq \mathcal{O}(V)$ that, for a small enough $\epsilon>0$, the operator
$$P_j(t):=\int_{|z|=\epsilon} (z-1/2+\mathbbm{B}_{\mu_j(t)}+\mathbbm{A}_{\mu_j(t)}\tau_t)^{-1}\rd z,$$
is the Riesz projection onto $\ker(-1/2+\mathbbm{B}_{\mu_j(t)}+\mathbbm{A}_{\mu_j(t)}\tau_t)$. We take a non-zero element $g_j(s_0)\in \ker(-1/2+\mathbbm{B}_{\mu_j(s_0)}+\mathbbm{A}_{\mu_j(s_0)}\tau_{s_0})$ and extend to a function $g_j=g_j(t)$ by
$$g_j(t):=P_j(t)g_j(s_0)\in\ker(-1/2+\mathbbm{B}_{\mu_j(t)}+\mathbbm{A}_{\mu_j(t)}\tau_t) .$$
Since $g_j(s_0)\neq 0$, $g_j(t)\neq 0$ in a neighborhood of $s_0$. 
\end{proof}

\begin{remark}
It follows from self-adjointness of $L^\Omega_{b,\tau_t}$ that $\mu_{jk}(V_k\cap [0,1])\subseteq \R$.
\end{remark}

We let $\mathrm{sign}\;:\R\setminus \{0\}\to \{-1,1\}$ denote the sign function. For a holomorphic function $\mathfrak{h}$ and a $t$ in its domain of definition, we let $\mathrm{ord}_t(\mathfrak{h}):=\inf\{k:\mathfrak{h}^{(k)}(t)\neq 0\}\in \N\cup\{\infty\}$ denote the order of $\mathfrak{h}$ at $t$. Note that $\mathrm{ord}_t(\mathfrak{h})=\infty$ if and only if $\mathfrak{h}\equiv0$. If $\mathfrak{h}$ is holomorphic at $s_0$ and vanishes there to odd order, we say that {\bf $\mathfrak{h}$ is odd at $s_0$}. The following proposition is an immediate consequence of Theorem \ref{parametrizingmu} and Proposition \ref{computationaflwo}.

\begin{prop}
\label{sfbymuj}
Take a $\mu\in \R\setminus \sigma_D$ and suppose that $(\tau_t)_{t\in [0,1]}\subseteq \Psi^0(\partial\Omega)^{s.a.}$ is a holomorphic path. Let $0=t_0<t_1<...<t_{M-1}<t_M=1$ be a partition of $[0,1]$, $(V_k)_{k=1}^M$ open neighborhoods of $[t_{k-1},t_k]$ in $\C$ and $(\mu_{jk})_{j=1}^{N_k}\subseteq \mathcal{O}(V_k)$ be holomorphic functions as in Theorem \ref{parametrizingmu}. Define the sets
$$Z_{jk}:=\left\{t\in [t_{k-1},t_k]:\mbox{$\mu_{jk}-\mu$ is odd at $t$}\right\}.$$ 
The spectral flow of $(L^\Omega_{b,\tau_t}-\mu)_{t\in [0,1]}$ can be expressed as 
$$\mathrm{sf}(L^\Omega_{b,\tau_t}-\mu)=\sum_{k=1}^M\sum_{j=1}^{N_k}\sum_{t\in Z_{jk}} \lim_{\epsilon\to 0} \mathrm{sign}\;\left(\mu_{jk}'(t+\epsilon)\right). $$
\end{prop}

\begin{remark}
The appearance of the small $\epsilon$ takes the possibility of $\mu_{jk}'(t)=0$ into account and corrects this problem because with a small enough $\epsilon$ the number $\mu_{jk}'(t+\epsilon)$ is not only non-zero, due to the analyticity of $\mu_{jk}$, but carries the same sign as $\mu_{jk}^{(k)}(t)$, where $k=\mathrm{ord}_t(\mu_{jk}-\mu)$ if $\mu_{jk}'(t)=0$ and measures the direction of flow when $\mu_{jk}'(t)$ blows up.
\end{remark}

We now turn to the proof of Theorem \ref{localizintobo} (see page \pageref{localizintobo}). The heart of the proof lies in the next two propositions and summarized below in Remark \ref{summinglocalizintobo}. Let $(\tau_t)_{t\in [0,1]}\subseteq \Psi^0(\partial\Omega)^{s.a.}$ be a holomorphic path. Define the holomorphic function 
$$\pmb{g}\in \mathcal{O}(U\times (\C\setminus \sigma_L)\times \C^\times), \quad\mbox{by}\quad \pmb{g}(t,\mu,\lambda):=\det\,\!_{2d}(1-2\lambda^{-1}\mathbbm{B}_\mu-2\lambda^{-1}\mathbbm{A}_\mu\tau(t)).$$
Recall the notation $\pmb{f}(t,\mu)=\pmb{g}(t,\mu,1)$ from the proof of Theorem \ref{parametrizingmu}. By construction, for any $\mu\notin\sigma_L$ and $t$, 
$$\sigma(\mathbbm{B}_\mu+\mathbbm{A}_\mu\tau(t)))\setminus \{0\}=\{\lambda\in \C^\times: \pmb{g}(t,\mu,\lambda/2)=0\}.$$

\begin{prop}
\label{famiimpl}
The family $(\mu_{jk})_{j=1}^{N_k}\subseteq \mathcal{O}(V_k)$ constructed in Theorem \ref{parametrizingmu} satisfies that 
$$\mu_{jk}'(t)=-\frac{\partial_t \pmb{f}(t,\mu_{jk}(t))}{\partial_\mu \pmb{f}(t,\mu_{jk}(t))}, \quad \mbox{for}\quad t\in V_k\cap [0,1].$$
If $\partial_\mu \pmb{f}(t,\mu_{jk}(t))=0$, the right hand side is made sense of through a finite part value.
\end{prop}

\begin{proof}
For notational simplicity we drop $k$ from the notation. We start under the assumption that $\partial_\lambda \pmb{g}(t,\mu_j(t),1)\neq 0$. The implicit function theorem allows us to parametrize $\pmb{g}(t,\mu,\lambda)=0$ locally as $\lambda=\lambda(t,\mu)$. Assuming $\partial_\mu\lambda(t,\mu_j(t))\neq 0$, the implicit function theorem allows us to parametrize $\lambda(t,\mu)=1$ locally as $\mu=\mu(t)$. For a suitable choice of $\lambda$, we can do so using $\mu(t)=\mu_j(t)$. It follows that
\begin{align*}
\mu_j'(t)=-\frac{\partial_t \lambda(t,\mu_j(t))}{\partial_\mu \lambda(t,\mu_j(t))}=-\frac{\partial_t \pmb{g}(t,\mu_j(t),1)}{\partial_\mu \pmb{g}(t,\mu_j(t),1)}&=-\frac{\partial_t \pmb{f}(t,\mu_j(t))}{\partial_\mu \pmb{f}(t,\mu_j(t))}.
\end{align*}
If $t$ is such that $\partial_\lambda \pmb{g}(t,\mu_j(t),1)= 0$, analyticity of all functions involved guarantees that we can carry out the same proof once perturbing $t$ by a small $\epsilon>0$ and letting $\epsilon\to 0$ -- giving the finite part value of $-\frac{\partial_t \pmb{f}(t,\mu_j(t))}{\partial_\mu \pmb{f}(t,\mu_j(t))}$.
\end{proof}

\begin{prop}
\label{focompo}
The partial derivatives of $\pmb{f}$ are given by
$$\begin{cases}
\partial_t\pmb{f}(t,\mu)=\tra_{L^2(\partial\Omega)}\left(\partial_t\Gamma(\tau_t,\mu)\Gamma(\tau_t,\mu)^{2d-1}(1+\Gamma(\tau_t,\mu))^{-1}\right)\cdot \det\,\!_{2d}(1+\Gamma(\tau_t,\mu))\\
\\
\partial_\mu \pmb{f}(t,\mu)=\tra_{L^2(\partial\Omega)}\left(\partial_\mu\Gamma(\tau_t,\mu)\Gamma(\tau_t,\mu)^{2d-1}(1+\Gamma(\tau_t,\mu))^{-1}\right)\cdot \det\,\!_{2d}(1+\Gamma(\tau_t,\mu))\end{cases}.$$
\end{prop}

\begin{proof}
Let $\alpha\mapsto A(\alpha)$ be a holomorphic function $U\to \mathcal{L}^{2d}(L^2(\partial\Omega))$ for an open neighborhood $U\subseteq \C$, e.g. $\Gamma(\tau_t,\mu)$ as a function of $t$ or of $\mu$. By construction, see \cite[Chapter 9]{simon}, $\det_{2d}(1+A(\alpha))=\det(1+R_{2d}(A(\alpha)))$ where $\det$ denotes the Fredholm determinant and 
$$R_{2d}(A(\alpha))=(1+A(\alpha))\exp\left(\sum_{j=1}^{2d-1}(-1)^j\frac{A(\alpha)^j}{j}\right)-1\in \mathcal{L}^1(L^2(\partial\Omega)).$$
A direct computation shows 
\begin{align*}
\frac{\rd}{\rd \alpha}\det(1+R_{2d}(A(\alpha)))&=\tra_{L^2(\partial\Omega)}\left(\frac{\rd R_{2d}(A(\alpha))}{\rd \alpha}(1+R_{2d}(A(\alpha)))^{-1}\right)\det(1+R_{2d}(A(\alpha)))\\
&=\tra_{L^2(\partial\Omega)}\left(\frac{\rd A(\alpha)}{\rd \alpha}A(\alpha)^{2d-1}(1+A(\alpha))^{-1}\right)\det\!\,_{2d}(1+A(\alpha)).
\end{align*}
\end{proof}

\begin{remark}
\label{summinglocalizintobo}
The two Propositions \ref{famiimpl} and \ref{focompo} together with Proposition \ref{sfbymuj} imply Theorem \ref{localizintobo} with 
$$Z_\mu(\tau)=\left\{t\in [0,1]: -1\in \sigma(\Gamma(\mu,\tau_t))\quad\mbox{and}\quad \frac{\partial_t \pmb{f}(t,\mu)}{\partial_\mu \pmb{f}(t,\mu)}\quad\mbox{is odd at $t$} \right\}$$
\end{remark}

\large
\section{Some observations on the spectral flow}
\normalsize

In the previous section we showed that the spectral flow can be defined. In this section we will show that in certain special cases the spectral flow can be computed.

\subsection{Monotonicity of spectral flow and asymptotics of flow}
\label{monoto}

\begin{thm}
\label{monotonethm}
Let $\mu\in \R\setminus \sigma_D$. Suppose that $(\tau_t)_{t\in [0,1]}\in \Psi^0(\partial\Omega)^{s.a.}$ is continuous in the $\Bo(H^{3/2}(\partial\Omega),H^{1/2}(\partial\Omega))$-norm. If $\tau_1\geq \tau_0$ as operators on $L^2(\partial\Omega)$ then 
$$\mathrm{sf}(L^\Omega_{b,\tau_t}-\mu)_{t\in [0,1]}\geq 0.$$
Moreover, $\mathrm{sf}(L^\Omega_{b,\tau_t}-\mu)_{t\in [0,1]}> 0$ if for a $q\in \N$ and a $\mu_0<\mu$ such that $\Lambda_q<\mu_0<\mu<\Lambda_{q+1}$, the following conditions are satisfied 
\begin{enumerate}
\item[A.)] $([\mu_0,\mu)\cap \sigma(L^\Omega_{b,\tau_0}))\setminus \sigma_D$ is non-empty.
\item[B.)] $\tau_1-\tau_0\geq (\mu-\mu_0)c^2\,\id_{L^2(\Omega)}$ as operators on $L^2(\partial\Omega)$ where 
$$c:=\sup_{t\in [0,1],\lambda\in [\mu_0,\mu]}\|\mathcal{B}_{\lambda}+\mathcal{A}_{\lambda}\tau_t\|_{L^2(\partial\Omega)\to L^2(\Omega)}.$$
\end{enumerate}
\end{thm}

We remark that $c<\infty$ by an argument similar to the proof of Lemma \ref{singcont}.

\begin{proof}
By Corollary \ref{sfcor} (see page \pageref{sfcor}), we can assume that $(\tau_t)_{t\in [0,1]}$ is the holomorphic path $\tau_t=(1-t)\tau_0+t\tau_1$. Take $(V_k)_{k=1}^M$ and $(\mu_{jk})_{j=1}^{N_k}$ as in Proposition \ref{sfbymuj}. It follows from Theorem \ref{parametrizingmu} that for any $k$ we can find $(v_{jk})_{j=1}^{N_k}\subseteq C^\infty(V_k\cap [0,1],H^2_A(\Omega))$ such that $\|v_{jk}(t)\|_{L^2(\Omega_A)}=1$ and $v_{jk}(t)\in \ker(L^\Omega_{b,\tau_t}-\mu_{jk}(t))$ for all $t\in V_k\cap [0,1]$. It follows from $\tau_t'=\tau_1-\tau_0$ and integration by parts that
\begin{align}
\nonumber
\mu_{jk}'(t)&=2\mathrm{Re}\left(\int_\Omega (\nabla+iA)v_{jk}(t)\cdot \overline{ (\nabla+iA)v_{jk}'(t)}\rd V+\int_{\partial\Omega}\tau_tv_{jk}(t)\cdot \overline{v_{jk}'(t)}\rd S\right)\\
\label{muprimbomp}
&\quad+\int_{\partial\Omega}\tau_t'v_{jk}(t)\cdot \overline{v_{jk}(t)}\rd S\\
\nonumber&=2\mathrm{Re}\langle L^\Omega_{b,\tau_t} v_{jk}(t),v_{jk}'(t)\rangle_{L^2(\Omega)}+\int_{\partial\Omega}(\tau_1-\tau_0)v_{jk}(t)\cdot \overline{v_{jk}(t)}\rd S\\
\nonumber&=\mathrm{Re}\left(\mu_{jk}(t)\frac{\rd}{\rd t}\|v_{jk}(t)\|^2_{L^2(\Omega)}\right)+\langle(\tau_1-\tau_0)v_{jk}(t),v_{jk}(t)\rangle_{L^2(\partial\Omega)}\\
\nonumber&=\langle(\tau_1-\tau_0)v_{jk}(t),v_{jk}(t)\rangle_{L^2(\partial\Omega)}\geq 0,
\end{align}
because $\tau_1-\tau_0\geq 0$ on $L^2(\partial\Omega)$. It follows immediately from Proposition \ref{sfbymuj} that $\mathrm{sf}(L^\Omega_{b,\tau_t}-\mu)_{t\in [0,1]}\geq 0$.

To prove the second statement of the theorem it suffices to construct a path of eigenvalues that crosses the point $\mu$. We choose a $\lambda \in [\mu_0,\mu)\cap \sigma(L^\Omega_{b,\tau_0})$ with $\lambda\notin \sigma_D$. For $\epsilon>0$ sufficiently small, Theorem \ref{parametrizingmu} shows that we can find a path $(\mu(t))_{t\in [0,\epsilon]}$ which is holomorphic, starting in $\mu(0)=\lambda$, and a $C^1$-path $(v(t))_{t\in [0,\epsilon]}$ such that $v(t)\in \Dom(L^\Omega_{b,\tau_t})$ satisfies $\|v(t)\|_{L^2(\Omega)}=1$ and $L^\Omega_{b,\tau_t}v(t)=\mu(t)v(t)$ for all $t\in[0,\epsilon]$. The computation \eqref{muprimbomp} implies that 
\begin{equation}
\label{mulowest}
\mu'(t)=\langle(\tau_1-\tau_0)v(t),v(t)\rangle_{L^2(\partial\Omega)}\geq (\mu-\mu_0)c^2\|\gamma_{\partial\Omega}v(t)\|_{L^2(\partial\Omega)}^2.
\end{equation}
It follows from Lemma \ref{muandlbtau} that for those $t\in [0,\epsilon]$ such that $\mu(t)\notin \sigma_D$ and $\lambda\leq \mu(t)\leq \mu$,  
$$1=\|v(t)\|_{L^2(\Omega)}=\|(\mathcal{B}_{\mu(t)}+\mathcal{A}_{\mu(t)}\tau_t)\gamma_{\partial\Omega}v(t)\|_{L^2(\Omega)}\leq c\|\gamma_{\partial\Omega}v(t)\|_{L^2(\partial\Omega)}.$$

Let $t_0$ denote the supremum over all $\epsilon$ for which $\mu$ can be extended to $[0,\epsilon]$. If $t_0<1$, $\mu(t)$ has a pole at $t_0$ and by Equation \eqref{mulowest} $\mu(t)\to +\infty$ as $t\to t_0$. In this case, $\mathrm{sf}(L^\Omega_{b,\tau_t}-\mu)_{t\in [0,1]}\geq 1$. Assume therefore that $\mu(t)$ can be continued to $[0,1]$. The continuation argument, and the assumption $\lambda\notin \sigma_D$, shows that $\mu$ satisfies $\mu'(t)\neq 0$ for $t$ in a dense open set hence the set of $t$ with $\mu(t)\in \sigma_D$ has measure $0$. We deduce $\mu'(t)\geq \mu-\mu_0$ for almost all $t\in [0,1]$ such that $\mu(t)\leq \mu$. In particular, $\mu(1)\geq \mu>\lambda=\mu(0)$. The inequality $\mathrm{sf}(L^\Omega_{b,\tau_t}-\mu)_{t\in [0,1]}\geq 1$ follows.
\end{proof}

\begin{remark}
The condition A.) in Theorem \ref{monotonethm} can in practice be verified for $\lambda\notin\sigma_D$ using for instance the Kato-Temple inequality (see \cite{harell,katinww,templeineq}) stating that $(\mu_0,\mu)\cap \sigma(L^\Omega_{b,\tau})$ is non-empty if there is a unit vector $\psi\in \Dom(L^\Omega_{b,\tau})$ such that, with $\eta:=\mathfrak{q}^\Omega_{b,\tau}[\psi]$ and $\epsilon:=\sqrt{\|L^\Omega_{b,\tau}\psi\|^2_{L^2(\Omega)}-\eta^2}$, $\epsilon^2<(\mu-\eta)(\eta-\mu_0)$.
\end{remark}

\begin{cor}
\label{weyllawforpsido}
Take $\tau,\tau'\in \Psi^0(\partial \Omega)^{s.a.}$ where $\tau$ is elliptic and strictly positive on $L^2(\partial\Omega)$. Define the elliptic first order operator 
$$D_{\tau,\tau'}(\mu):=\left(\mathbbm{A}_\mu^{-1}(1/2-\mathbbm{B}_\mu)-\tau'\right) \tau^{-1},$$ 
and the family $\tau(t):=\tau'+t\tau$. For $\mu\in \R\setminus \sigma_D$, 
$$\mathrm{sf}(L^\Omega_{b,\tau(t)}-\mu)_{t\in [0,\gamma]}=\#\left([0,\gamma]\cap \sigma(D_{\tau,\tau'}(\mu))\right).$$
In particular, if $D_{\tau,\tau'}(\mu)$ is self-adjoint in a neighborhood of a $\mu_0\in \R\setminus \sigma_L$,
$$\mathrm{sf}(L^\Omega_{b,\tau(t)}-\mu_0)_{t\in [0,\gamma]}=\frac{\gamma^{2d-1}}{(2\pi)^{2d-1}}\int_{S^*\partial\Omega}\sigma_0(\tau)^{2d-1}\rd S_{S^*\partial\Omega}+O(\gamma^{2d-2}), \quad \mbox{as}\quad \gamma\to \infty.$$
\end{cor}

\begin{proof}
Using Lemma \ref{muandlbtau}, Proposition \ref{sfbymuj} and the proof of Theorem \ref{monotonethm}, it follows that for $\mu \notin \sigma_D$
\begin{align*}
\mathrm{sf}(L^\Omega_{b,\tau(t)}-\mu)_{t\in [0,\gamma]}&=\#\left\{t\in [0,\gamma]: 1/2\in \sigma(\mathbbm{B}_\mu+\mathbbm{A}_\mu(\tau'+t\tau))  \right\}\\
&=\#\left\{t\in [0,\gamma]: 2t\in \sigma\left(\mathbbm{A}_\mu^{-1}\tau^{-1}-2\mathbbm{A}_\mu^{-1}\mathbbm{B}_\mu\tau^{-1}+2\tau'\tau^{-1}\right)  \right\}\\
&=\#\left([0,\gamma]\cap \sigma(D_{\tau,\tau'}(\mu))\right).
\end{align*}
The operator $D_{\tau,\tau'}(\mu)$ has a positive principal symbol, namely $\sigma_1(D_{\tau,\tau'}(\mu))=\sigma_0(\tau)^{-1}$. So, if $D_{\tau,\tau'}(\mu)$ is self-adjoint it is also bounded from below by G\aa rding's inequality. It follows from the Weyl law for the self-adjoint first order operator $D_{\tau,\tau'}(\mu)$, see \cite[Theorem 29.1.5]{horiv}, that 
\begin{align*}
\#\big([0,\gamma]\cap &\sigma(D_{\tau,\tau'}(\mu))\big)=N(0,\gamma;D_{\tau,\tau'}(\mu))+O(1)\\
&=\frac{\gamma^{2d-1}}{(2\pi)^{2d-1}}\int_{S^*\partial\Omega}\sigma_1(D_{\tau,\tau'}(\mu))^{-2d+1}\rd S_{S^*\Omega}+O(\gamma^{2d-2}), \quad \mbox{as}\quad \gamma\to \infty.
\end{align*}
The above identities make sense and hold true also when $\mu\in \sigma_D\setminus \sigma_L$, after choosing a self-adjoint parametrix of $\mathbbm{A}_\mu$.
\end{proof}

\subsection{Calculations on the disc}
\label{disccompsubse}

We consider the case where $K$ is the closed unit disc in $\R^2$, so $\Omega:=\R^2\setminus K=\{z\in \C: \,|z|>1\}$ and $\partial \Omega=S^1$. There is a $U(1)$-action on $L^2(\Omega)$ decomposing 
$$L^2(\Omega)\cong \bigoplus_{n\in \Z} L^2((1,\infty),r \rd r).$$
The unitary $U=\oplus_{n\in \Z}U_n:L^2(\Omega)\to \bigoplus_{n\in \Z} L^2((1,\infty),r \rd r)$ implementing the isomorphism is defined from $U_nf(r):=\int_0^{2\pi} f(r\e^{i\theta})\e^{-in\theta}\rd \theta$. 

To simplify the discussion of Landau-Robin operators, we assume that the pseudo-differential operator $\tau\in \Psi^0(S^1)$ is $U(1)$-equivariant. In particular, there is a sequence $(\tau_n)_{n\in \Z}\in \ell^\infty(\Z)$ such that for $g=\sum_{n\in \Z} g_n\e^{in\theta}\in L^2(S^1)$, $\tau g(\theta)=\sum_{n\in \Z} \tau_n g_n\e^{in\theta}$. Not all sequences in $\ell^\infty(\Z)$ arise from a $U(1)$-equivariant pseudodifferential operator, for a characterization of such sequences, see \cite{agran84,melopsu,grig78}. The algebra of $U(1)$-equivariant pseudo-differential operators on the circle is commutative. For any $m$, the norm on the $U(1)$-equivariant pseudo-differential operators coming from $\mathcal{B}(H^s(S^1),H^{s-m}(S^1))$ does not depend on $s\in \R$ and is given by $\|\tau\|_{\Bo(H^s(S^1),H^{s-m}(S^1))}=\sup_{n\in \Z} (1+|n|)^{-m}|\tau_n|$. In fact, the arguments in this subsection go through for any operator $\tau$ with $\tau g(\theta)=\sum_{n\in \Z} \tau_n g_n\e^{in\theta}$ constructed from a sequence $(\tau_n)_{n\in \Z}\in \ell^\infty(\Z)$.

When $\tau$ is $U(1)$-equivariant, the unitary transformation $U$ satisfies
$$UL^\Omega_{b,\tau}U^*=\bigoplus_{n\in \Z} H_n(b,\tau_n),$$
where we take a graph closure of the direct sum on the right hand side and $H_n(b,t)$ is the differential expression
\begin{align*}
H_n(b,t)&=-\frac{\rd^2}{\rd r^2}-\frac{1}{r}\frac{\rd }{\rd r}+\left(\frac{n}{r}-\frac{br}{2}\right)^2,\quad\mbox{equipped with the domain} \\ 
&\Dom(H_n(b,t)):=\left\{w:\; H_n(b,t)w\in L^2((1,\infty),r \rd r)\;\mbox{and}\;w'(1)-tw(1)=0\right\}.
\end{align*}

It can be shown that the differential equation $H_n(b,t)w=\lambda w$ in $(1,\infty)$ has a one-dimensional space of $L^2$-solutions spanned by the function
$$w_{\lambda,n,b}(r)=r^{-1}W_{\frac{n+\lambda b^{-1}}{2},\frac{n}{2}}\left(\frac{br^2}{2}\right).$$
Here $W_{\kappa,\mu}$ denotes the Whittaker functions, see \cite[Chapter 13]{abramosteg}. It follows that 
$$\sigma(H_n(b,t))=\{\lambda\in \R: w_{\lambda,n,b}'(1)-tw_{\lambda,n,b}(1)=0\}.$$

\begin{prop}
When $\Omega$ is the complement of the closed unit disc in $\R^2$, we can describe the spectrum of a Landau-Robin operator defined from a $U(1)$-equivariant pseudo-differential operator $\tau$ as
$$\sigma(L^\Omega_{b,\tau})=\overline{\bigcup_{n\in \Z} \left\{\lambda\in \R: w_{\lambda,n,b}'(1)-\tau_nw_{\lambda,n,b}(1)=0\right\}}.$$
\end{prop}

The analysis of the boundary operators $\mathbbm{A}_\mu$ and $\mathbbm{B}_\mu$ simplifies greatly as they lie in the commutative algebra of $U(1)$-equivariant pseudo-differential operators; this fact is seen from the definition of their integral kernels, see Proposition \ref{fundsolprop}. The Fourier modes of the operators $\mathbbm{A}_\mu$ and $-1/2+\mathbbm{B}_\mu$, respectively, are given by 
\begin{align}
\label{anbnet}
\mathfrak{a}_n(\mu,b)&:=\frac{1}{4\pi}\int_0^{2\pi}\int_0^\infty \e^{in\theta}\mathfrak{e}(t,\theta,\mu,b)\rd t\rd \theta,\\
\nonumber
\\
\nonumber
\mathfrak{b}_n(\mu,b)&:=\frac{1}{8\pi}\int_0^{2\pi}\int_0^\infty \e^{in\theta}\mathfrak{e}(t,\theta,\mu,b)\left(i\sin(\theta)+(\cos(\theta)-1)\coth(t)\right)\rd t\rd \theta,\\
\nonumber
\\
\nonumber
&\mbox{where}\quad \mathfrak{e}(t,\theta,\mu,b):=\frac{\mathrm{exp}\left(ib\sin(\theta)-\frac{1-\cos(\theta)}{2}\coth(t)+\frac{\mu t}{b}\right)}{\sinh(t)}.
\end{align}
Both integrals are interpreted as principal value integrals. It follows that for any self-adjoint $U(1)$-equivariant pseudo-differential operators $\tau$ and $\tau'$ on $S^1$ the operator $D_{\tau,\tau'}(\mu)$ of Corollary \ref{weyllawforpsido} is indeed self-adjoint. Using the explicit formulas of Proposition \ref{fundsolprop} and the characterization of Lemma \ref{muandlbtau}, we infer the following proposition from Equation \eqref{anbnet}.

\begin{prop}
Let $\Omega$ be the complement of the closed unit disc in $\R^2$ and $\tau$ a $U(1)$-equivariant pseudo-differential operator on $S^1$ with Fourier modes $(\tau_n)_{n\in \Z}$. We have the equality
$$\sigma(L^\Omega_{b,\tau})\setminus \sigma_D=\overline{\bigcup_{n\in \Z} \bigg\{\mu\in \C\setminus \sigma_D: \mathfrak{b}_n(\mu,b)+\tau_n\mathfrak{a}_n(\mu,b)=0\bigg\}}.$$
\end{prop}

We conclude this paper with a result concerning the multiplicities of the Landau level $\Lambda_q$ as an eigenvalue for the exterior of the disc. We can treat the Landau levels, i.e. the essential spectrum, in this particular case due to the simple geometry of the circle. At the Landau level $\Lambda_q=(2q-1)b$, 
$$w_{\Lambda_q,n,b}(r)=(q-1)!(-1)^{q-1}\left(\frac{b}{2}\right)^{\frac{n+1}{2}}r^n\e^{-br^2/4} L^n_{q-1}\left(\frac{br^2}{2}\right),$$
where $L_{q-1}^n$ denotes the Laguerre polynomials. As such, $\Lambda_q$ is an eigenvalue for $L^\Omega_{b,\tau}$ if and only if for an $n$ the following polynomial equation which is linear in $\tau_n$ and of order $q$ in $b$ holds:
\begin{equation}
\label{bcforll}
(2n-b-2\tau_n)L^n_{q-1}\left(\frac{b}{2}\right)+2bL^{n+1}_{q-2}\left(\frac{b}{2}\right)=0.
\end{equation}
Here we have used the identity $L^{n+1}_{q-2}=(L^n_{q-1})'$. We remark that the zeroes of the Laguerre polynomials are simple, hence it never holds that $L^n_{q-1}(b/2)=0$ and $L^{n+1}_{q-2}(b/2)=0$ simultaneously. For $q=1$, this equation is equivalent to the linear relation $\tau_n=n-b/2$. For  $L^n_{q-1}(b/2)\neq 0$ we can solve the equation \eqref{bcforll}. We define a sequence giving a solution to Equation \eqref{bcforll} (whenever it exists) as
\begin{equation}
\label{solvingbcfor}
T_n(q,b):=n-\frac{b}{2}+b\ell_n(q,b),\quad\mbox{where}\quad \ell_n(q,b):=\begin{cases} 0, &\mbox{for}\;L^n_{q-1}\left(\frac{b}{2}\right)=0\\
\,\\
 \frac{L^{n+1}_{q-2}\left(\frac{b}{2}\right)}{L^n_{q-1}\left(\frac{b}{2}\right)}, & \mbox{for}\;L^n_{q-1}\left(\frac{b}{2}\right)\neq0. \end{cases}
\end{equation}
Equation \eqref{bcforll} shows how the Weyl law from Corollary \ref{weyllawforpsido} comes into play: as $\tau_n$ increases, eigenvalues will cross Landau levels in a linear fashion. From the particular form of $T_n$ from Equation \eqref{solvingbcfor} we can deduce properties about the multiplicities of $\Lambda_q$ as an eigenvalue of $L^\Omega_{b,\tau}$.

\begin{prop}
Let $\Omega$ be the complement of the closed unit disc in $\R^2$. The sequence $(T_n(q,b))_{n\in \Z}$ is the sequence of Fourier modes associated with an elliptic $U(1)$-equivariant pseudo-differential operator $T=T(q,b)\in \Psi^1(S^1)^{U(1)}$. Moreover, if $\tau\in \Psi^t(S^1)^{U(1)}$ for $t<1$ and $\Lambda_q$ is an eigenvalue of $L^\Omega_{b,\tau}$, $\Lambda_q$ has at most finite multiplicity satisfying the upper estimate
\begin{equation}
\label{estimasol}
\dim(L^\Omega_{b,\tau}-\Lambda_q)\leq C(q,b)\#\bigg\{n:\; \tau_n\in\left[n-\frac{b}{2}-\frac{d(q,b)}{|n|},n-\frac{b}{2}+\frac{d(q,b)}{|n|}\right]\bigg\},
\end{equation}
for some constants $C(q,b),d(q,b)>0$.
\end{prop}

\begin{proof}
The expression 
$$L^n_{q-1}(x)=\sum_{i=0}^{q-1}\begin{pmatrix} q-1+n\\ q-1-i\end{pmatrix}\frac{(-x)^i}{i!},$$
for the Laguerre polynomials shows  
$$\ell_n(q,b)=\frac{\sum_{i=1}^{q-1}c_{i-1,q}(n)b^{i-1}}{\sum_{i=0}^{q-1}c_{i,q}(n)b^i}$$
where 
$$c_{i,q}(n):=\begin{pmatrix} q-1+n\\ q-1-i\end{pmatrix}\frac{(-1)^i}{2^ii!}.$$
The expression $\ell_n(q,b)$ being a rational function in $n$, has an asymptotic expansion as $|n|\to \infty$ with leading order contribution $(q-1)n^{-1}$. Existence of the asymptotic expansion implies that $\ell_n(q,b)$ defines a pseudo-differential operator of order $-1$ by \cite{agran84} or \cite{melopsu}. Therefore $T$ is a pseudo-differential operator of order $1$ whose principal symbol takes the values $\pm 1$. We deduce ellipticity of $T$. 

If $\tau\in \Psi^t(S^1)^{U(1)}$ for $t<1$, then $|\tau_n|\lesssim |n|^t$ and there is at most a finite number of $n$ for which $\tau_n=T_n$. The number of $n$ for which $\tau_n=T_n$ and $L^n_{q-1}\left(\frac{b}{2}\right)\neq0$ holds coincides with $\dim(L^\Omega_{b,\tau}-\Lambda_q)$ by Equation \eqref{bcforll} and \eqref{solvingbcfor}. This finite number of solutions is bounded by the right hand side of the estimate \eqref{estimasol} since $|\ell_n(q,b)|\lesssim |n|^{-1}$.
\end{proof}

\section*{{\bf Acknowledgements}}
Thanks to Grigori Rozenblum for carefully reading the manuscript and making several useful suggestions and remarks. Thanks to Mikael Persson Sundqvist for helpful discussions and sharing the unpublished note \cite{mpsnote}, both providing the inspiration for writing Section \ref{disccompsubse}.

\end{document}